\documentclass[a4paper,10pt]{amsart}

\usepackage{amsmath}
\usepackage{amsthm}
\usepackage{amssymb}
\usepackage{amsfonts}
\usepackage{enumitem} 
\usepackage{mathrsfs} 
\usepackage{xcolor}
\usepackage[bookmarks=true,hyperindex,pdftex,colorlinks,citecolor=blue]{hyperref}%
\usepackage[a4paper,lmargin=3cm,rmargin=3cm,tmargin=4cm,bmargin=4cm,marginparwidth=2.8cm,marginparsep=1mm]{geometry} 
\usepackage{marginnote} 
\usepackage{soul} 

\usepackage{xr}

\definecolor{ballblue}{rgb}{0.13, 0.67, 0.8}

\DeclareMathOperator{\lspan}{span}                          
\DeclareMathOperator{\conv}{conv}                           
\DeclareMathOperator{\supp}{supp}                           
\DeclareMathOperator{\Lip}{Lip}                             
\DeclareMathOperator{\Br}{Br}                               

\newcommand{\NN}{\mathbb{N}}                                
\newcommand{\RR}{\mathbb{R}}                                

\newcommand{\abs}[1]{\left|{#1}\right|}                     
\newcommand{\pare}[1]{\left({#1}\right)}                    
\newcommand{\set}[1]{\left\{{#1}\right\}}                   
\newcommand{\norm}[1]{\left\|{#1}\right\|}                  
\newcommand{\dual}[1]{{#1}^\ast}                            
\newcommand{\duality}[1]{\left<{#1}\right>}                 
\newcommand{\cl}[1]{\overline{#1}}                          
\newcommand{\restrict}{\mathord{\upharpoonright}}           
\newcommand{\lipfree}[1]{\mathcal{F}({#1})}                 
\newcommand{\lipnorm}[1]{\norm{#1}_L}                       
\newcommand{\meas}[1]{\mathcal{M}({#1})}                    
\newcommand{\wt}[1]{\widetilde{#1}}                         
\newcommand{\bwt}[1]{\beta\wt{#1}}                          
\newcommand{\pp}{\mathfrak{p}}                              
\newcommand{\opr}[1]{\mathcal{M}_{\mathrm{op}}(#1)}         
\newcommand{\norming}[1]{\mathcal{N}(#1)}                   

\newcommand{\rcomp}[1]{#1^\mathcal{R}}                    
\newcommand{\ucomp}[1]{#1^\mathcal{U}}                    

\makeatletter
\newcommand{\labeltext}[2]{%
  \@bsphack
  \csname phantomsection\endcsname 
  \def\@currentlabel{#1}{\label{#2}}%
  \@esphack
}
\makeatother

\theoremstyle{plain}
\newtheorem{theorem}{Theorem}[section]
\newtheorem{lemma}[theorem]{Lemma}
\newtheorem{corollary}[theorem]{Corollary}
\newtheorem{proposition}[theorem]{Proposition}
\newtheorem*{claim}{Claim}

\theoremstyle{definition}
\newtheorem*{definition*}{Definition}
\newtheorem{definition}[theorem]{Definition}
\newtheorem{example}[theorem]{Example}
\newtheorem{question}[theorem]{Question}

\theoremstyle{remark}
\newtheorem{remark}[theorem]{Remark}

\numberwithin{equation}{section}

\begin{document}

\title{Convex integrals of molecules in Lipschitz-free spaces}

\author[R. J. Aliaga]{Ram\'on J. Aliaga}
\address[R. J. Aliaga]{Instituto Universitario de Matem\'atica Pura y Aplicada,
Universitat Polit\`ecnica de Val\`encia,
Camino de Vera S/N,
46022 Valencia, Spain}
\email{raalva@upvnet.upv.es}

\author[E. Perneck\'a]{Eva Perneck\'a}
\address[E. Perneck\'a]{Faculty of Information Technology, Czech Technical University in Prague, Th\'akurova 9, 160 00, Prague 6, Czech Republic}
\email{perneeva@fit.cvut.cz}

\author[R. J. Smith]{Richard J. Smith}
\address[R. J. Smith]{School of Mathematics and Statistics, University College Dublin, Belfield, Dublin 4, Ireland}
\email{richard.smith@maths.ucd.ie}

\date{}

\begin{abstract}
We introduce convex integrals of molecules in Lipschitz-free spaces $\mathcal{F}(M)$ as a continuous counterpart of convex series considered elsewhere, based on the de Leeuw representation.
Using optimal transport theory, we show that these elements are determined by cyclical monotonicity of their supports, and that under certain finiteness conditions they agree with elements of $\mathcal{F}(M)$ that are induced by Radon measures on $M$, or that can be decomposed into positive and negative parts. We also show that convex integrals differ in general from convex series of molecules.
Finally, we present some standalone results regarding extensions of Lipschitz functions which, combined with the above, yield applications to the extremal structure of $\mathcal{F}(M)$. In particular, we show that all elements of $\mathcal{F}(M)$ are convex series of molecules when $M$ is uniformly discrete and identify all extreme points of the unit ball of $\mathcal{F}(M)$ in that case.
\end{abstract}


\subjclass[2020]{Primary 46B20; Secondary 46E15, 49Q22.}

\keywords{Lipschitz-free space, de Leeuw representation, Lipschitz function, optimal transport, cyclical monotonicity, extreme point}

\maketitle


\section{Introduction}

This paper is framed within the theory of Lipschitz and Lipschitz-free spaces. In the following, $M$ will be a complete metric space with metric $d$, where a fixed base point $0\in M$ has been designated. Given a function $f:M\to\RR$, its (optimal) Lipschitz constant is denoted by
$$
\lipnorm{f} = \sup\set{\frac{f(x)-f(y)}{d(x,y)} \,:\, x\neq y\in M} .
$$
We write $\Lip(M)$ for the space of all real-valued Lipschitz functions on $M$, and $\Lip_0(M)$ for the subspace consisting of all $f\in\Lip(M)$ such that $f(0)=0$. It is straightforward to check that $(\Lip_0(M),\lipnorm{\cdot})$ is a Banach space. For every $x\in M$, the evaluation functional $\delta(x):f\in\Lip_0(M)\mapsto f(x)$ is an element of the dual $\dual{\Lip_0(M)}$, and $\delta:M\to\dual{\Lip_0(M)}$ is an isometric embedding. The \emph{Lipschitz-free space} is then defined as
$$
\lipfree{M} = \cl{\lspan}\,\delta(M) \subset \dual{\Lip_0(M)} ,
$$
and it is actually an isometric predual of $\Lip_0(M)$. We recommend Weaver's monograph \cite{Weaver2} for reference, where Lipschitz-free spaces are called Arens-Eells spaces and $\lipfree{M}$ is denoted \AE$(M)$.

Lipschitz and Lipschitz-free spaces are closely connected to the theory of optimal transport, so it is unsurprising that they appear in various branches of mathematics. Their main applications to Banach space theory are often stated to be in non-linear geometry, and many of them can already be found in the seminal papers by Godefroy and Kalton \cite{GK,Kalton04}; for a survey of those and their implications, see \cite{Godefroy_survey,GLZ}. In addition, Lipschitz and Lipschitz-free spaces have become handy tools for the construction of examples and counterexamples in linear Banach space theory; see for instance \cite[Example 4.4]{AP_rmi}, \cite[Section 3]{Veeorg}, or \cite[Section 2]{KaasikVeeorg}.

The most important elements of $\lipfree{M}$ are arguably the so-called \emph{(elementary) molecules}
$$
m_{xy} = \frac{\delta(x)-\delta(y)}{d(x,y)} \in S_{\lipfree{M}}
$$
for pairs of points $x\neq y\in M$. They represent the action of taking incremental quotients of functions $f\in\Lip_0(M)$. In particular, the set of all elementary molecules on $M$ is norming for $\Lip_0(M)$ and so its closed convex hull is the whole unit ball $B_{\lipfree{M}}$. It follows that every $m\in S_{\lipfree{M}}$ can be written, for any $\varepsilon>0$, as a series of molecules
\begin{equation}
\label{eq:series_of_molecules}
m = \sum_{n=1}^\infty a_nm_{x_ny_n}
\end{equation}
where $\sum_n\abs{a_n}\leq 1+\varepsilon$ (see e.g. \cite[Lemma 2.1]{AP_rmi}); we may and will assume $a_n>0$ for all $n$ by swapping $x_n$ and $y_n$, removing null terms, and/or repeating terms in finite sums.
Clearly $\sum_n a_n\geq 1$ for any representation of the form \eqref{eq:series_of_molecules}, so the natural question arises whether the optimal value $\sum_n a_n=1$ can be attained. The answer is negative for general $m$, but it is positive in particular cases. For instance, if $m$ has finite support (i.e. it is a finite sum of evaluation functionals) then it can even be written as a finite convex combination of molecules which, in fact, codifies the solution to a finite optimal transport problem represented by $m$ (see e.g. \cite[Proposition 3.16]{Weaver2}).

In \cite{ARZ}, the term \emph{(sum of a) convex series of molecules} was used for those elements $m\in S_{\lipfree{M}}$ that could be written in the form \eqref{eq:series_of_molecules} with $\sum_n a_n=1$. These elements are particularly easy to manipulate in convexity arguments, and this can be exploited to characterise precisely when they satisfy some isometric properties. For instance, \cite[Theorem 3.1]{ARZ} describes which of these elements are points of G\^ateaux differentiability of the norm of $\lipfree{M}$, and it is easy to see that any extreme point of $B_{\lipfree{M}}$ that is a convex series of molecules must in fact be an elementary molecule (see \cite[Remark 3.4]{APPP}).

This paper will be focused on a generalisation of that notion. Note that the expression \eqref{eq:series_of_molecules} describes a discrete, necessarily countable sum. We wish to consider a continuous counterpart with the hope that the same type of convexity arguments will still apply to them. Intuitively, this should be achieved by substituting the infinite convex sum by a probability integral of the form $m=\int m_{xy}\,d\mu(x,y)$. A natural name for those $m$ is then \emph{convex} (or \emph{probability}) \emph{integrals of molecules}. While it is possible to construct them explicitly as Bochner integrals in $\lipfree{M}$, it will be more convenient to do so indirectly through a method for integral representation of functionals on $\Lip_0(M)$ due to de Leeuw \cite{deLeeuw}.
De Leeuw representations have already proved their usefulness in questions about extremal structure \cite{Aliaga,AG,AP_rmi,Weaver95} and duality \cite{deLeeuw,Weaver96} in Lipschitz-free spaces.

The structure of the paper is as follows. We first spend the rest of this section establishing the notation and required background facts, including a primer on the relationship between Lipschitz-free spaces and optimal transport theory. After that, in Section \ref{sec:optimal} we introduce the integral representation due to de Leeuw and use it to properly define convex integrals of molecules. We show that, among all de Leeuw representations of functionals, convex integrals of molecules are determined by cyclical monotonicity of their support (Theorem \ref{th:cyclic_monotone_optimal}).

In Section \ref{sec:positive} we show that convex integrals of molecules are bona fide generalisations of optimal couplings from optimal transport theory, and that any functional that is either representable as a measure on $M$, or what we call ``majorisable on annuli'' (which includes all functionals expressible as the difference of two positive ones), is a convex integral of molecules (Theorems \ref{th:major_radon_wt} and \ref{th:majorizable_convex_integrals}). In particular, all elements of $\lipfree{M}$ are convex series of molecules when $M$ is uniformly discrete. We also derive some consequences for general expressions as convex series of molecules.

Next, in Section \ref{sec:negative} we provide some counterexamples to representability as convex integrals of molecules. Most significantly, if $M$ contains isometrically a subset of $\RR$ with positive measure then there are elements of $\lipfree{M}$ that are not convex integrals of molecules (Theorem \ref{th:isometric_fat_cantor_example}). We also show that convex series and integrals of molecules are different notions in general (Proposition \ref{pr:integral_not_series}). The overall message of this section is that the picture regarding these notions is subtle and complex.

Finally, we devote Section \ref{sec:extensions} to the proof of some standalone results about the relationship between metric alignment and extensions of Lipschitz functions, and we apply them in Section \ref{sec:extreme} to the study of the extremal structure of the unit ball $B_{\lipfree{M}}$. A well-known conjecture states that all extreme points of $B_{\lipfree{M}}$ must be elementary molecules (see e.g. \cite{Aliaga}). We show that this holds for extreme points that are representable as a convex integral of molecules (Theorem \ref{th:extreme_wt}) which, in particular, proves the conjecture when $M$ is uniformly discrete. We also obtain information on the structure of faces of $B_{\lipfree{M}}$ (see Theorem \ref{th:alignment_compact}) that, in turn, furnishes a new and more straightforward proof of the characterisation of extreme molecules given originally in \cite{AP_rmi}.

Convex series and integrals of molecules are particular types of de Leeuw representations. This paper is part of a larger programme to study the properties of general de Leeuw representations, and is expanded upon in \cite{APS_future}.

\subsection{Notation and preliminaries on Lipschitz-free spaces}

All topological spaces considered hereafter will be Hausdorff. Recall that the Stone-\v{C}ech compactification $\beta X$ of a metrisable (or, more generally, Tychonoff) space $X$ is the topologically unique compact space containing $X$ as a dense subset and such that any continuous mapping from $X$ into a compact space admits a continuous extension to $\beta X$. We will identify any such function on $X$ with its continuous extension to $\beta X$ without labelling it differently.
In particular, when $M$ is a metric space, for $p\in M$ and $\xi\in\beta M$ we will denote by $d(p,\xi)$ or $d(\xi,p)$ the value at $\xi$ of the continuous extension of the one-variable mapping $x\mapsto d(p,x)$ from $M$ to $\beta M$, with range in $[0,\infty]$. We emphasise that the expression $d(\xi,\eta)$ is not meaningful for general $\xi,\eta\in\beta M$, as it is not always possible to extend the two-variable function $d:M\times M\to\RR$ to $\beta M\times\beta M$ continuously: that expression is only meaningful when one of the coordinates is fixed and belongs to $M$.

By a measure on a Hausdorff space $X$ we will always mean a Borel, countably additive measure, and by a \emph{Radon measure} we will mean one that is regular, tight, and finite (i.e. has finite total variation). If $X$ is separable and completely metrisable, then all finite measures on $X$ are Radon (see \cite[Theorems 1.4.8 and 7.1.7]{Bogachev}).
The Banach space of all Radon measures on $X$ endowed with the total variation norm will be denoted by $\meas{X}$; recall that $\meas{X}=\dual{C(X)}$ when $X$ is compact. The \emph{support} $\supp(\mu)$ of a measure $\mu$ is defined as the set of all $x\in X$ such that $\abs{\mu}(U)>0$ for all open neighbourhoods $U$ of $x$. It is always a closed set, and it is also separable if $X$ is metrisable and $\mu$ is finite.
If $\mu$ is Radon then it is concentrated on its support, and $\abs{\mu}(\supp(\mu))=\norm{\mu}$ (see \cite[Theorem 7.2.9]{Bogachev}).

The point mass on $x\in X$ will be denoted as $\delta_x\in\meas{X}$ (note that this is different from the evaluation functional $\delta(x) \in \lipfree{M}$, $x \in M$, introduced above). For Borel $E\subset X$ and $\mu\in\meas{X}$, the restriction $\mu\restrict_E\in\meas{X}$ will be defined by $\mu\restrict_E(A)=\mu(A\cap E)$ for any Borel $A\subset X$.
For $\mu,\lambda\in\meas{X}$, the notation $\mu\ll\lambda$ will stand for ``$\mu$ is absolutely continuous with respect to $\lambda$'', and $\mu\perp\lambda$ for ``$\mu$ and $\lambda$ are mutually singular''.
If $X$ and $Y$ are Hausdorff spaces and $f:X\to Y$ is Borel, then for any positive Borel measure $\mu$ on $X$ the \emph{pushforward measure} $f_\sharp\mu$ is the positive Borel measure on $Y$ defined by $f_\sharp\mu(E)=\mu(f^{-1}(E))$ for all Borel $E\subset Y$. It has the property that $\int_Y g\,d(f_\sharp\mu)=\int_X (g\circ f)\,d\mu$ for every Borel function $g:Y\to\RR$ such that $g\circ f$ is $\mu$-integrable. If moreover $\mu$ is Radon and $f$ is continuous then $f_\sharp\mu$ is Radon as well.

If $M$ is a metric space, the \emph{metric segment} between points $p,q\in M$ is the set
$$
[p,q] = \set{x\in M \,:\, d(p,x)+d(x,q)=d(p,q)} .
$$
If $x\in[p,q]$, we will also say that $p,x,q$ are \emph{metrically aligned} or that $(p,x,q)$ is a \emph{metric triple}.

As is customary, we will write $B_X$ and $S_X$ for the unit ball and unit sphere of a Banach space $X$, which will be usually one of the spaces $\lipfree{M}$ and $\Lip_0(M)$ introduced above. In an abuse of notation, we will also write $B_{\Lip(M)}$ and $S_{\Lip(M)}$ for the set of all $f:M\to\RR$ with $\lipnorm{f}\leq 1$ and $\lipnorm{f}=1$, respectively. Recall that $\dual{\lipfree{M}}=\Lip_0(M)$ canonically, and the weak$^\ast$ topology induced by $\lipfree{M}$ agrees on $B_{\Lip_0(M)}$ with the topology of pointwise convergence. Changing the base point of $M$ results in linearly isometric Lipschitz and Lipschitz-free spaces, respectively.

One key feature of Lipschitz-free spaces is their linearisation property: given any pair of metric spaces $M,N$ with base points $0_M,0_N$, it is possible to extend any Lipschitz map $\psi:M\to N$ such that $\psi(0_M)=0_N$ to a bounded linear operator $\widehat{\psi}:\lipfree{M}\to\lipfree{N}$ such that $\widehat{\psi}\circ\delta_M=\delta_N\circ\psi$; that is, $\widehat{\psi}$ is an extension of $\psi$ if we identify $M,N$ with their images $\delta_M(M),\delta_N(N)$ in their respective Lipschitz-free spaces. Moreover, $\|\widehat{\psi}\|$ equals the Lipschitz constant of $\psi$. In particular, if $\psi$ is an isometric embedding then so is $\widehat{\psi}$. It follows that the Lipschitz-free space over a subset $N\subset M$ containing $0$ can be identified with a subspace of $\lipfree{M}$, namely $\lipfree{N}=\cl{\lspan}\,\delta(N)$. The \emph{support} $\supp(m)$ of a functional $m\in\lipfree{M}$ is then defined as the smallest closed subset of $M$ such that $m\in\lipfree{\supp(m)\cup\set{0}}$. The existence of such a set is not immediate; the proof is given in \cite{AP_rmi,APPP}. The set $\supp(m)$ is always separable, and the value of $\duality{m,f}$ for $f\in\Lip_0(M)$ depends only on the restriction of $f$ to $\supp(m)$. A point $x\in M$ belongs to $\supp(m)$ if and only if every neighbourhood of $x$ contains the support of some $f\in\Lip_0(M)$ such that $\duality{m,f}\neq 0$ \cite[Proposition 2.7]{APPP}. Recall also that elements of finite support are precisely the finite linear combinations of evaluation functionals, and therefore dense in $\lipfree{M}$.

A functional $m\in\lipfree{M}$ is said to be \emph{positive} if $\duality{m,f}\geq 0$ whenever $f\in\Lip_0(M)$ satisfies $f\geq 0$ pointwise.
As there is a unique pointwise largest function in $B_{\Lip_0(M)}$, namely $x\mapsto d(x,0)$, we always have $\norm{m}=\duality{m,d(\cdot,0)}$ for positive $m$.

\subsection{Preliminaries on optimal transport}
\label{subsec:optimal_transport}

The classical Monge-Kantorovich transportation problem involves moving a certain amount of mass from its current position into a new one at the least possible cost. The current and desired mass distributions are modelled as Radon probability measures $\mu$ and $\nu$ on a metric space $M$, respectively. We assume moreover that $\mu$ and $\nu$ have \emph{finite first moment}, that is
\begin{equation}
\label{eq:finite_first_moment}
\int_M d(x,0)\,d\abs{\mu}(x) < \infty
\end{equation}
for some (and hence all) $0\in M$, and likewise for $\nu$. Possible ways of transferring mass from $\mu$ to $\nu$ are represented by \emph{couplings}, or \emph{transport plans}, from $\mu$ to $\nu$: these are the probability measures $\pi\in\meas{M\times M}$ whose marginals are $\mu$ and $\nu$, that is $(\pp_1)_\sharp\pi=\mu$ and $(\pp_2)_\sharp\pi=\nu$, where $\pp_i:M\times M\to M$ are the coordinate projections.
The cost of transferring mass from location $x$ to location $y$ is assumed to be proportional to the distance $d(x,y)$. The goal is then to find an optimal coupling that minimises the total transportation cost
$$
C(\pi)=\int_{M\times M}d(x,y)\,d\pi(x,y) .
$$
Note that the product measure $\mu\times\nu$ is a coupling between $\mu$ and $\nu$ and
\begin{align*}
C(\mu\times\nu) &\leq \int_{M\times M}(d(x,0)+d(y,0))\,d(\mu\times\nu)(x,y) \\
&= \norm{\nu}\int_M d(x,0)\,d\mu(x) + \norm{\mu}\int_M d(y,0)\,d\nu(y) < \infty
\end{align*}
by \eqref{eq:finite_first_moment}. Thus there are always valid transport plans, and the optimal cost is always finite.
Using a compactness argument, one can prove that there always exists a (not necessarily unique) optimal coupling that minimises the cost, and the Kantorovich duality theorem provides a formula for the optimal cost in terms of Lipschitz functions.

\begin{theorem}[Kantorovich duality theorem]
\label{th:kantorovich}
Let $\mu,\nu$ be Radon probability measures on $M$ with finite first moment, and let $\Pi\subset\meas{M\times M}$ be the set of all couplings from $\mu$ to $\nu$. Then
$$
\inf_{\pi\in\Pi}\,C(\pi) = \sup_{f\in B_{\Lip(M)}} \set{\int_M f\,d\mu - \int_M f\,d\nu}
$$
and the infimum is attained by some coupling $\pi\in\Pi$.
\end{theorem}

Optimal couplings can be characterised geometrically as follows. A subset $E$ of $M\times M$ is said to be \emph{cyclically monotonic} if it satisfies the following: for every choice of points $(x_1,y_1),\ldots,(x_n,y_n)\in E$ it holds that
\begin{equation}
\label{eq:cyclical_monotonicity}
d(x_1,y_1)+d(x_2,y_2)+\ldots+d(x_n,y_n) \leq d(x_1,y_2)+d(x_2,y_3)+\ldots+d(x_n,y_1) .
\end{equation}
Then any optimal coupling between measures on $M$ is concentrated on a cyclically monotonic subset of $M\times M$. Conversely, any Radon probability measure that is concentrated on a cyclically monotonic subset of $M\times M$ is an optimal coupling between its marginals (see e.g. \cite{ST}). In particular, restrictions of optimal couplings are again optimal couplings.

The \emph{1-Wasserstein space} of $M$ is defined as the metric space $(\mathcal{W}_1(M),d_{\mathcal{W}_1})$ of all Radon probability measures on $M$ with finite first moment endowed with the Wasserstein distance $d_{\mathcal{W}_1}(\mu,\nu)$ given by the optimal transport cost from $\mu$ to $\nu$. To see that $d_{\mathcal{W}_1}$ is a metric, note that any measure $\mu\in\meas{M}$ with finite first moment induces a functional $\widehat{\mu}\in\lipfree{M}$ by integration, given by
\begin{equation}
\label{eq:functional_induced_by_measure}
\duality{\widehat{\mu},f} = \int_M f\,d\mu
\end{equation}
for $f\in\Lip_0(M)$ \cite[Proposition 4.4]{AP_measures}. Thus Theorem \ref{th:kantorovich} tells us that the optimal cost of the transport problem between probability measures $\mu,\nu\in\meas{M}$ is precisely
$$
d_{\mathcal{W}_1}(\mu,\nu)=\norm{\widehat{\mu}-\widehat{\nu}}_{\lipfree{M}}
$$
(note that the supremum in Theorem \ref{th:kantorovich} may be taken over $f\in B_{\Lip_0(M)}$ instead of $f\in B_{\Lip(M)}$, as adding a constant to $f$ does not change the resulting value).
Thus $d_{\mathcal{W}_1}$ is indeed a metric, and the mapping $\mu\mapsto\widehat{\mu}$ is an isometry from $\mathcal{W}_1(M)$ into $\lipfree{M}$. Its image is precisely the set $\cl{\conv}\,\delta(M)$, as follows easily from the density of finitely supported measures (cf. \cite[Corollary 4.16]{AP_measures}).

A standard reference for the theory of optimal transport is \cite{Villani}, although it focuses on the case where $M$ is separable. For an easy proof of the Kantorovich theorem for arbitrary $M$, see \cite{Edwards}.

\section{de Leeuw representations}
\label{sec:optimal}

We start by introducing a general representation of elements of $\dual{\Lip_0(M)}$ that can be traced back to K. de Leeuw \cite{deLeeuw}. Let us denote
$$
\wt{M} = \set{(x,y)\in M\times M \,:\, x\neq y}
$$
with the inherited metric product topology.
The \emph{de Leeuw transform} is the map $\Phi:\Lip_0(M)\to C_b(\wt{M})$ (the space of bounded continuous real-valued functions on $\wt{M}$) given by
$$
(\Phi f)(x,y) = \frac{f(x)-f(y)}{d(x,y)}
$$
for $(x,y)\in\wt{M}$. Note that $\norm{\Phi f}_\infty = \lipnorm{f}$ by definition, so $\Phi$ is a linear isometry into. By identifying each function in $C_b(\wt{M})$ with its unique continuous extension to $\bwt{M}$, we shall identify the spaces $C_b(\wt{M})$ and $C(\bwt{M})$ and thus consider $\Phi$ as an isometry into $C(\bwt{M})$. It then follows that the adjoint operator $\dual{\Phi}:\meas{\bwt{M}}\to\dual{\Lip_0(M)}$ is a non-expansive onto quotient map. Hence, for every $\phi\in\dual{\Lip_0(M)}$ there exist measures $\mu\in\meas{\bwt{M}}$ such that $\dual{\Phi}\mu=\phi$, meaning that
$$
\duality{f,\phi} = \int_{\bwt{M}}(\Phi f)\,d\mu
$$
for all $f\in\Lip_0(M)$. Any such $\mu$ will be referred to as a \emph{de Leeuw representation} of $\phi$.
The easiest examples of such representations are the Dirac deltas $\delta_{(x,y)}$ for $(x,y)\in\wt{M}$, which represent molecules $m_{xy}$ as $(\Phi f)(x,y)=\duality{m_{xy},f}$.

As we will be dealing with points in the compactification $\bwt{M}$, we need a way to identify their ``coordinates''. We will write $\pp:\bwt{M}\to\beta M\times \beta M$ for the continuous extension of the identity mapping $\wt{M}\to \beta M\times \beta M$, and $\pp_1,\pp_2:\bwt{M}\to\beta M$ for the projections onto the first and second coordinates, i.e. the extensions of $\pp_1:(x,y)\mapsto x$ and $\pp_2:(x,y)\mapsto y$. For a set $E\subset\bwt{M}$, we will also use the notation
$$
\pp_s(E)=\pp_1(E)\cup\pp_2(E)
$$
for the combined projection. We make two remarks here. First, this definition of $\pp_i$ agrees with the notation used in Section \ref{subsec:optimal_transport} for elements of $\wt{M}$, so there shall be no ambiguity. Second, note that we may use any compactification of $M$ in place of $\beta M$ in the definition of $\pp$ and $\pp_i$.

We should mention here that the framework of de Leeuw representations can be developed in an equivalent way by considering finite, regular, \emph{finitely additive} measures on $\wt{M}$ instead of Radon measures on $\bwt{M}$, which avoids ambiguity and the use of compactifications altogether. This approach is taken e.g. in \cite{HOP22,Veeorg2}.

Since $\dual{\Phi}$ is non-expansive, any de Leeuw representation $\mu\in\meas{\bwt{M}}$ of a functional $\phi\in\dual{\Lip_0(M)}$ has norm $\norm{\mu}\geq\norm{\phi}$. Our focus will be on norm-optimal representations, i.e. those with the least possible norm $\norm{\mu}=\norm{\phi}$. Every functional admits norm-optimal representations, and it is not hard to see that they can even be chosen to be positive \cite[Proposition 3]{Aliaga}. By an \emph{optimal (de Leeuw) representation} we will mean one that is both norm-optimal and positive. The collection of all positive measures on $\bwt{M}$ that are optimal de Leeuw representations (of whichever functional in $\dual{\Lip_0(M)}$ they are representing) will be denoted as
$$
\opr{\bwt{M}} = \set{\mu\in\meas{\bwt{M}} \,:\, \mu\geq 0 \text{ and } \norm{\mu}=\norm{\dual{\Phi}\mu}} .
$$
The following statement summarises some useful properties of this set.

\begin{proposition}
\label{pr:opr_facts}
\
\begin{enumerate}[label={\upshape{(\alph*)}}]
\item For any $\phi\in\dual{\Lip_0(M)}$ there is $\mu\in\opr{\bwt{M}}$ such that $\dual{\Phi}\mu=\phi$.
\item If $\mu\in\opr{\bwt{M}}$ then $c\cdot\mu\in\opr{\bwt{M}}$ for every $c\geq 0$.
\item If $\mu\in\opr{\bwt{M}}$ and $\lambda\in\meas{\bwt{M}}$ satisfies $0\leq\lambda\leq\mu$, then $\lambda\in\opr{\bwt{M}}$.
\item If $\mu\in\opr{\bwt{M}}$ and $E$ is a Borel subset of $\bwt{M}$ then $\mu\restrict_E\in\opr{\bwt{M}}$.
\end{enumerate}
\end{proposition}

\begin{proof}
(a) is \cite[Proposition 3]{Aliaga} and (b) is trivial. For part (c) notice that
$$
\norm{\dual{\Phi}\mu} \leq \norm{\dual{\Phi}\lambda} + \norm{\dual{\Phi}(\mu-\lambda)} \leq \norm{\lambda} + \norm{\mu-\lambda} = \norm{\mu} = \norm{\dual{\Phi}\mu}
$$
and therefore $\norm{\lambda}=\norm{\dual{\Phi}\lambda}$ as well. (d) is a particular case of (c).
\end{proof}

In the sequel, we will need two facts about the supports of de Leeuw representations. First, it is intuitively clear that the support of a representation needs to somehow cover the support of the represented functional. We state this principle as follows; for the proof, we refer to \cite{Aliaga}.

\begin{proposition}[cf. {\cite[Lemma 8]{Aliaga}}]
\label{pr:aliaga_lemma_8}
If $m\in\lipfree{M}$, then $\supp(m)\subset\pp_s(\supp(\mu))$ for any de Leeuw representation $\mu$ of $m$.
\end{proposition}

For the second fact, let us fix some additional notation. Given $\phi\in\dual{\Lip_0(M)}$, denote by
$$
\norming{\phi} = \set{f\in S_{\Lip_0(M)} \,:\, \duality{f,\phi}=\norm{\phi}}
$$
the set of norming functions for $\phi$. This set may be empty in general, but it never is when $\phi\in\lipfree{M}$. The following statement details how norming functions relate to optimal de Leeuw representations.

\begin{lemma}
\label{lm:norming_phi1}
Let $\mu\in \meas{\bwt{M}}$ be a positive de Leeuw representation of $\phi\in\dual{\Lip_0(M)}$. Then
\[
\set{f\in S_{\Lip_0(M)} \,:\, \Phi f(\zeta)=1 \textup{ for all }\zeta\in \supp(\mu)} \subset \norming{\phi},
\]
and the two sets are equal if $\mu$ is optimal. If the set on the left-hand side is non-empty then $\mu$ is optimal.
\end{lemma}

\begin{proof}
If $f\in S_{\Lip_0(M)}$ is such that $\Phi f=1$ on $\supp(\mu)$, then
$$
\|\phi\|\geq\duality{f,\phi} = \int_{\bwt{M}}(\Phi f)\,d\mu = \int_{\supp(\mu)}(\Phi f)\,d\mu = \norm{\mu} \geq \norm{\phi}
$$
and so $\|\mu\|=\|\phi\|$ and $f\in\norming{\phi}$. Conversely, suppose that $\mu$ is optimal and $f\in\norming{\phi}$ . Then $\norm{\mu} = \norm{\phi} = \duality{f,\phi} = \int_{\bwt{M}} (\Phi f)\,d\mu$ and, since $\mu$ is positive and $\abs{\Phi f}\leq 1$ pointwise, we must have $\Phi f=1$ $\mu$-almost everywhere on $\bwt{M}$, and thus on all of $\supp(\mu)$ by continuity.
\end{proof}

We are often interested specifically in de Leeuw representations $\mu\in\meas{\bwt{M}}$ of elements of $\lipfree{M}$ instead of its bidual $\dual{\Lip_0(M)}$. Sadly, there is in general no easy way to tell them apart based on simple features of $\mu$, but we can identify several sufficient and necessary conditions. We will now single out two such necessary conditions that will be relevant in Section \ref{sec:positive}. The first fact is that $\mu$ cannot contain any mass concentrated at a single point in the ``diagonal'' of $\wt{M}$.

\begin{lemma}
\label{lm:measure_over_a_point}
If $\mu\in\meas{\bwt{M}}$ is a norm-optimal (not necessarily positive) representation of an element of $\lipfree{M}$ then $\mu(\pp^{-1}(p,p))=0$ for every $p\in M$.
\end{lemma}

\begin{proof}
Let $\mu\in\meas{\bwt{M}}$ be such that $\dual{\Phi}\mu\in\lipfree{M}$ and $\|\mu\|=\|\dual{\Phi}\mu\|$, and fix $p\in M$.
Recall that $\phi\in\Lip_0(M)^*$ is a \emph{derivation at $p$} if $\duality{f,\phi}=0$ for any $f\in\Lip_0(M)$ that is constant in a neighbourhood of $p$ (see \cite[Section 2.5]{AP_measures} and \cite[Section 7.5]{Weaver2}).
Because $\Phi f(\zeta)=0$ whenever $\pp(\zeta)=(p,p)$ and $f\in\Lip_0(M)$ is constant in some neighbourhood of $p$, we have that $-\Phi^*\left(\mu\restrict_{\pp^{-1}(p,p)}\right)$ is a derivation at $p$. Now we apply \cite[Proposition 2.11]{AP_measures}, which asserts that derivations and weak$^*$ continuous functionals on $\Lip_0(M)$ are in $\ell_1$-sum with each other, and obtain
$$
\norm{\Phi^*\left(\mu\restrict_{\bwt{M}\setminus\pp^{-1}(p,p)}\right)} = \norm{\Phi^*\mu-\Phi^*\left(\mu\restrict_{\pp^{-1}(p,p)}\right)} = \norm{\Phi^*\mu}+\norm{\Phi^*\left(\mu\restrict_{\pp^{-1}(p,p)}\right)} .
$$
The norm-optimality of $\mu$ finally implies that
$$
\norm{\mu} = \norm{\dual{\Phi}\mu} \leq \norm{\dual{\Phi}\left(\mu\restrict_{\bwt{M}\setminus\pp^{-1}(p,p)}\right)} \leq \norm{\mu\restrict_{\bwt{M}\setminus\pp^{-1}(p,p)}}
$$
and therefore $\mu(\pp^{-1}(p,p))=0$.
\end{proof}

The second fact is that $\mu$ must be concentrated ``away from infinity''. To make this statement precise, let us denote
\begin{equation}
\label{eq:Mr_definition}
\rcomp{M} = \set{\xi\in\beta M \,:\, d(\xi,0) < \infty}
\end{equation}
for the set of elements of $\beta M$ that do not lie at infinity; here, $d(\xi,0)$ stands for the evaluation at $\xi$ of the continuous extension of $d(\cdot,0)$ to $\beta M$. Equivalently, $\rcomp{M}$ consists of the elements of $\beta M$ that are limits of bounded nets in $M$.

Let us remark here that in \cite{AP_measures} and \cite{APS_future} $\rcomp{M}$ is defined with respect to the so-called uniform compactification $\ucomp{M}$ of $M$, rather than $\beta M$. However, in this paper none of the arguments involving $\rcomp{M}$ rely on this distinction, hence so as to avoid overcomplicating matters, we have chosen not to use $\ucomp{M}$.

So, for $\rcomp{M}$ as in \eqref{eq:Mr_definition} we have:

\begin{proposition}[cf. {\cite[Proposition 7]{Aliaga} and \cite[Section 2.3]{APS_future}}]
\label{pr:finite_concentration}
If $\mu\in\opr{\bwt{M}}$ is such that $\dual{\Phi}\mu\in\lipfree{M}$, then $\mu$ is concentrated on $\pp^{-1}(\rcomp{M}\times\rcomp{M})$.
\end{proposition}

\noindent This is an extension of \cite[Proposition 7]{Aliaga}, which only covers the case when $\rcomp{M}=M$. This happens precisely when $M$ is \emph{proper}, i.e. its closed balls are compact. In fact, Proposition \ref{pr:finite_concentration} is stated without proof at the end of p. 8 in \cite{Aliaga}. The argument used in the proof of Lemma 6 and Proposition 7 in \cite{Aliaga} already yields Proposition \ref{pr:finite_concentration} with almost no change other than replacing $M\times M$ with $\rcomp{M}\times\rcomp{M}$, so we will omit the proof here.

\subsection{Integrals of molecules}

Let us now consider a simple sufficient condition for $\mu\in\meas{\bwt{M}}$ to represent a functional in $\lipfree{M}$: when $\mu$ is concentrated on $\wt{M}$,
then $\dual{\Phi}\mu$ is always a weak$^\ast$ continuous functional. This is proved e.g. in \cite[Lemma 4.36]{Weaver2} or \cite[Proposition 4]{Aliaga}. We now provide a different argument that shows that, in fact, de Leeuw representations concentrated on $\wt{M}$ correspond to expressions of $\dual{\Phi}\mu$ as an integral of molecules in a literal sense.

\bigskip 

\begin{proposition}
\label{pr:wt_bochner}
If $\mu\in\meas{\bwt{M}}$ then
$$
\dual{\Phi}(\mu\restrict_{\wt{M}}) = \int_{\wt{M}}m_{xy}\,d\mu(x,y)
$$
as a Bochner integral in $\lipfree{M}$.
\end{proposition}

Note that restriction to $\wt{M}$ is a valid operation as it is a Borel (in fact, $G_\delta$) subset of $\bwt{M}$.

\begin{proof}
Assume without loss of generality that $\mu$ is Radon and concentrated on $\wt{M}$, i.e. $\mu=\mu\restrict_{\wt{M}}$.
We first check that the integral in the statement is a valid Bochner integral. Since $\norm{m_{xy}}=1$ and $\norm{\mu}<\infty$, it is enough to verify that the mapping $(x,y)\mapsto m_{xy}$ is measurable, i.e. that it is weakly measurable and almost separably valued (see e.g. \cite[Propositions 5.1 and 5.2]{BL}). The former means that the mapping $(x,y)\mapsto\duality{m_{xy},f}$ is measurable for each $f\in\dual{\lipfree{M}}=\Lip_0(M)$, which is obvious as that mapping is precisely $\Phi f$. For the latter, notice that $\mu$ is concentrated on $\supp(\mu)\cap\wt{M}$ by Radonness. This set agrees with the support of $\mu$ taken with $\wt{M}$ as the ambient space, therefore it is separable since $\wt{M}$ is metrisable. We conclude that the integral is valid and represents an element of $\lipfree{M}$. To verify the equality, we only need to check that
$$
\duality{\int_{\wt{M}}m_{xy}\,d\mu(x,y),f} = \int_{\wt{M}}\duality{m_{xy},f} d\mu(x,y) = \int_{\wt{M}}\Phi f(x,y)\,d\mu(x,y) = \duality{f,\dual{\Phi}\mu}
$$
for any $f\in\Lip_0(M)$.
\end{proof}

The converse of Proposition \ref{pr:wt_bochner} does not hold, in the sense that not every de Leeuw representation $\mu$ of an element of $\lipfree{M}$ is concentrated on $\wt{M}$; see \cite[Example 5]{Aliaga} for a simple counterexample where $\mu$ is concentrated entirely outside of $\wt{M}$. However, it is true that every $m\in\lipfree{M}$ admits at least one such representation. Indeed, recall that, for every $\varepsilon>0$, $m$ can be written as a series of molecules $m=\sum_na_nm_{x_ny_n}$ with $\sum_na_n\leq\norm{m}+\varepsilon$. That is, in fact, a discrete integral of molecules given by the de Leeuw representation
$$
\mu = \sum_{n=1}^\infty a_n\delta_{(x_n,y_n)}
$$
of $m$, which is positive, finite, discrete, and concentrated on $\wt{M}$. When $m\in S_{\lipfree{M}}$ is a convex series of molecules, then it admits one such representation $\mu$ that is moreover norm-optimal, i.e. $\sum_na_n=\norm{m}$; conversely, any such $\mu$ clearly corresponds to a convex series of molecules. In order to define their continuous counterparts, we now simply remove the requirement that $\mu$ be discrete.

\begin{definition}
\label{df:convex_integral}
An element of $\lipfree{M}$ will be called a \emph{convex integral of molecules} if it can be written as $\dual{\Phi}\mu$ for some $\mu$ belonging to the set
$$
\opr{\wt{M}} = \set{\mu\in\opr{\bwt{M}} \,:\, \text{$\mu$ is concentrated on $\wt{M}$}} .
$$
If $\mu$ is discrete, we will also use the term \emph{convex series of molecules}.
\end{definition}

This is a slight abuse of terminology as, rigorously, we should only speak of convex (or probability) integrals or series when $\norm{\mu}=1$, and thus only for elements of $S_{\lipfree{M}}$, but this notation will serve our purposes better. Note also that, by Proposition \ref{pr:opr_facts}(d), the measures in $\opr{\wt{M}}$ are just the restrictions of measures in $\opr{\bwt{M}}$ to $\wt{M}$.
Under the equivalent formulation in terms of finitely additive measures, (convex) integrals of molecules correspond precisely to those (optimal) de Leeuw representations that are countably additive, that is, to measures in $\meas{\wt{M}}$.

In general, not every convex integral of molecules is a convex series; we will provide a counterexample to this in Proposition \ref{pr:integral_not_series}. But both notions necessarily agree in situations when we can guarantee that $\mu\restrict_{\wt{M}}$ is discrete, e.g. when $M$ is countable. More generally:

\begin{proposition}
\label{pr:scattered_series}
If $M$ is scattered, then every convex integral of molecules on $M$ is also a convex series of molecules on $M$.
\end{proposition}

\begin{proof}
Suppose that $m\in\lipfree{M}$ has a de Leeuw representation $\mu\in\opr{\wt{M}}$. As in the proof of Proposition \ref{pr:wt_bochner}, $\supp(\mu)\cap\wt{M}$ agrees with the support of $\mu$ taken with $\wt{M}$ as the ambient space. But $\wt{M}$ is scattered as well, so it follows easily that $\supp(\mu)\cap\wt{M}$ is countable. Thus $\mu$ is discrete.
\end{proof}

We end this section by showing a simple criterion that characterises when an integral of molecules is convex: a positive measure on $\wt{M}$ is optimal if and only if it is concentrated on a cyclically monotonic set. This was already observed in \cite{ARZ}, under different notation, for discrete series of molecules, but an examination of the argument therein reveals that discreteness is not required.

\begin{theorem}
\label{th:cyclic_monotone_optimal}
Suppose that $\mu\in\meas{\bwt{M}}$ is positive and concentrated on $\wt{M}$. Then the following are equivalent:
\begin{enumerate}[label={\upshape{(\roman*)}}]
\item $\mu\in\opr{\wt{M}}$,
\item $\supp(\mu)\cap\wt{M}$ is cyclically monotonic,
\item $\mu$ is concentrated on a cyclically monotonic set.
\end{enumerate}
\end{theorem}

\begin{proof}
The equivalence is based on the following classical fact.

\begin{proposition}
\label{pr:cyclic_monotone_function}
A subset $E\subset\wt{M}$ is cyclically monotonic if and only if there exists $f\in B_{\Lip_0(M)}$ such that $\Phi f=1$ on $E$.
\end{proposition}

\noindent For a direct proof of this Proposition, we refer to the equivalence (i)$\Leftrightarrow$(iv) in \cite[Theorem 2.4]{ARZ}. Although the result from \cite{ARZ} is stated only for countable $E$, the same proof is valid for any set $E$ with no changes whatsoever.

Theorem \ref{th:cyclic_monotone_optimal} follows easily from Proposition \ref{pr:cyclic_monotone_function}. Indeed, suppose first that (i) is true, that is $\mu$ is optimal. Since $\dual{\Phi}\mu\in\lipfree{M}$ by Proposition \ref{pr:wt_bochner}, there exists some $f\in\norming{\dual{\Phi}\mu}$. Then $\Phi f=1$ on $\supp(\mu)$ by Lemma \ref{lm:norming_phi1}, thus $\supp(\mu)\cap\wt{M}$ is cyclically monotonic by Proposition \ref{pr:cyclic_monotone_function} and (ii) holds. The implication (ii)$\Rightarrow$(iii) is clear because $\mu$ is Radon. Finally, suppose that $\mu\in\meas{\bwt{M}}$ is positive and concentrated on a cyclically monotonic set $E\subset\wt{M}$, and let $f$ be given by Proposition \ref{pr:cyclic_monotone_function}. Then
$$
\norm{\mu} \geq \norm{\dual{\Phi}\mu} \geq \duality{\dual{\Phi}\mu,f} = \int_{\bwt{M}}(\Phi f)\,d\mu = \int_E(\Phi f)\,d\mu = \mu(E) = \norm{\mu}
$$
and thus $\norm{\mu}=\norm{\dual{\Phi}\mu}$. So (iii)$\Rightarrow$(i) is also true and the proof is complete.
\end{proof}

The following consequence of Theorem \ref{th:cyclic_monotone_optimal} is worth noting.

\begin{corollary}
\label{cr:metric_triples_pm}
Let $\mu\in\opr{\wt{M}}$. Suppose that $(x_0,x_1), (x_1,x_2), \ldots, (x_{n-1},x_n) \in\supp(\mu)\cap\wt{M}$. Then $\sum_{k=1}^n d(x_{k-1},x_k) = d(x_0,x_n)$. In particular, $x_1,\ldots,x_{n-1} \in [x_0,x_n]$.
\end{corollary}

\begin{proof}
By cyclical monotonicity we have
\begin{align*}
d(x_0,x_1) + d(x_1,x_2) + \ldots + d(x_{n-1},x_n) &= d(x_{n-1},x_n) + d(x_{n-2},x_{n-1}) + \ldots + d(x_0,x_1) \\
&\leq d(x_{n-1},x_{n-1}) + d(x_{n-2},x_{n-2}) + \ldots + d(x_1,x_1) + d(x_0,x_n) \\
&= d(x_0,x_n)
\end{align*}
and in particular, for each $k=1,\ldots,n-1$,
$$
d(x_0,x_k) + d(x_k,x_n) \leq d(x_0,x_1) + \ldots + d(x_{k-1},x_k) + d(x_k,x_{k+1}) + \ldots + d(x_{n-1},x_n) \leq d(x_0,x_n)
$$
thus all inequalities are equalities.
\end{proof}

This connection between optimal de Leeuw representations and cyclical monotonicity is generalised to certain elements of $\dual{\Lip_0(M)}$ (specifically those that ``avoid infinity'', see \cite{AP_measures} for the definition) in \cite{APS_future}.

\section{Convex integrals and functionals induced by measures}
\label{sec:positive}

It is reasonable to expect there to be a relationship between optimal de Leeuw representations and optimal couplings or transport plans, given that both are measures on $\wt{M}$ or $\bwt{M}$ that minimise some value. The next theorem exhibits an explicit correspondence between them and, at the same time, provides a large class of elements of $\lipfree{M}$ that can be represented as convex integrals of molecules. It turns out that every functional of the form $\widehat{\lambda}$ for $\lambda\in\meas{M}$, given by \eqref{eq:functional_induced_by_measure}, admits such a representation. In fact, there is a correspondence between optimal couplings (between the positive and negative parts of $\lambda$) and those de Leeuw measures in $\opr{\wt{M}}$ that satisfy additional finiteness conditions.

\begin{theorem}
\label{th:major_radon_wt}
Let $m\in\lipfree{M}$. Then the following are equivalent:
\begin{enumerate}[label={\upshape{(\roman*)}}]
\item $m=\widehat{\lambda}$ for some Radon measure $\lambda\in\meas{M}$,
\item $m$ is a convex integral of molecules with a representation $\mu\in\opr{\wt{M}}$ such that
\begin{equation}
    \label{eq:weighted finite}
    \int_{\widetilde{M}}\frac{1}{d(x,y)}\,d\mu(x,y)<\infty
\end{equation}
and
\begin{equation}
    \label{eq:marginals_finite_first_moment}
    \int_{\widetilde{M}}\frac{d(x,0)}{d(x,y)}\,d\mu(x,y)<\infty .
\end{equation}
\end{enumerate}
If the above hold, then $\mu$ can be chosen to satisfy $(\pp_1)_\sharp\mu\ll\lambda^+$ and $(\pp_2)_\sharp\mu\ll\lambda^-$ where $\lambda=\lambda^+-\lambda^-$ is the Jordan decomposition of $\lambda$. In particular, $(\pp_1)_\sharp\mu \perp (\pp_2)_\sharp\mu$.
\end{theorem}

Let us note that convex integrals of molecules do not need to admit any representation from $\opr{\wt{M}}$ satisfying conditions \eqref{eq:weighted finite} or \eqref{eq:marginals_finite_first_moment}. Indeed, there exist convex integrals of molecules such that all of their optimal representations concentrated on $\wt{M}$ fail one or both of these conditions (see Theorem \ref{th:majorizable_convex_integrals}, Example \ref{ex:not_majorisable} and Proposition \ref{pr:isometric_l1}).

\begin{proof}
First, assume that (ii) holds and $m=\dual{\Phi}\mu$.
As in the proof of Proposition \ref{pr:wt_bochner}, $\supp(\mu)\cap\wt{M}$ is separable and so we may assume that $M$ is separable by replacing it with the closure in $M$ of $\pp_s(\supp(\mu)\cap\wt{M})$.
Define a weighted measure $\nu$ on $\widetilde{M}$ by $d\nu(x,y)=d(x,y)^{-1}\,d\mu(x,y)$. By \eqref{eq:weighted finite}, $\nu$ is finite and therefore Radon (see e.g. \cite[Lemma 7.1.11]{Bogachev}). Then $\lambda_i=(\pp_i)_\sharp \nu$ for $i=1,2$ are Radon measures on $M$, hence $\lambda=\lambda_1-\lambda_2$ is a signed Radon measure on $M$. By adding a multiple of $\delta_0$, we may also assume that $\lambda(\{0\})=0$. Moreover, condition \eqref{eq:marginals_finite_first_moment} and the finiteness of $\mu$ ensure that both $\lambda_1$ and $\lambda_2$ have finite first moments. Indeed,
$$
\int_M d(x,0)\,d\lambda_1(x) = \int_M d(x,0)\,d((\pp_1)_\sharp\nu)(x) = \int_{\wt{M}} d(x,0)\,d\nu(x,y) = \int_{\wt{M}} \frac{d(x,0)}{d(x,y)}\,d\mu(x,y)<\infty
$$
and
\begin{align*}
\int_M d(y,0)\,d\lambda_2(y) &= \int_M d(y,0)\,d((\pp_2)_\sharp\nu)(y)=\int_{\wt{M}} \frac{d(y,0)}{d(x,y)}\,d\mu(x,y)\\
&\leq \int_{\wt{M}} \frac{d(x,y)+d(x,0)}{d(x,y)}\,d\mu(x,y)=\|\mu\|+ \int_{\wt{M}} \frac{d(x,0)}{d(x,y)}\,d\mu(x,y) <\infty.
\end{align*}
So, by \cite[Proposition 4.4]{AP_measures}, $\lambda$ induces a functional $\widehat{\lambda}$ in $\lipfree{M}$, and because
for every $f\in \Lip_0(M)$ we have
\begin{align*}
\int_M f\, d\lambda&=\int_M f\, d\lambda_1 - \int_M f\, d\lambda_2\\
&=\int_{\widetilde{M}} \frac{f(\pp_1(x,y))}{d(x,y)}\, d\mu(x,y) - \int_{\widetilde{M}} \frac{f(\pp_2(x,y))}{d(x,y)}\, d\mu(x,y)\\
&=\duality{\Phi^*\mu,f}=\duality{m,f} ,
\end{align*}
we conclude that $\widehat{\lambda}=m$.

Conversely, assume that (i) holds, i.e. $m=\widehat{\lambda}$ for some $\lambda\in\meas{M}$. Then $\lambda^+, \lambda^-\in\meas{M}$ have finite first moment by \cite[Proposition 4.4]{AP_measures}. Adding a multiple of $\delta_0$ to $\lambda$ does not change the effect of integrating functions $f\in\Lip_0(M)$ against $\lambda$ because $f(0)=0$, so we may assume that $\lambda(M)=0$, i.e. $\norm{\lambda^+}=\norm{\lambda^-}$. Further, by multiplying by a constant factor we assume without loss of generality that $\norm{\lambda^+}=\norm{\lambda^-}=1$. We can now apply the Kantorovich duality theorem (Theorem \ref{th:kantorovich}) to conclude that there exists an optimal coupling $\pi\in\meas{M\times M}$ from $\lambda^+$ to $\lambda^-$ that attains the optimal transport cost
$$
\int_{M\times M}d(x,y)\,d\pi(x,y) = d_{\mathcal{W}_1}(\lambda^+,\lambda^-) = \norm{\widehat{\lambda^+}-\widehat{\lambda^-}}_{\lipfree{M}} = \norm{\widehat{\lambda}}_{\lipfree{M}} = \norm{m} .
$$
Let us check that $\pi$ is concentrated on $\wt{M}$. Indeed, denoting $\Delta=(M\times M)\setminus\wt{M}$ for the diagonal, we clearly have $(\pp_1)_\sharp(\pi\restrict_\Delta)=(\pp_2)_\sharp(\pi\restrict_\Delta)$, and this measure is a lower bound for $(\pp_1)_\sharp\pi=\lambda^+$ and $(\pp_2)_\sharp\pi=\lambda^-$. But $\lambda^+\perp\lambda^-$, so we conclude that $(\pp_1)_\sharp(\pi\restrict_\Delta)=0$ and thus $\pi\restrict_\Delta=0$. Therefore we can define a positive Radon measure $\mu\in\meas{\wt{M}}$ by
$$
d\mu(x,y) = d(x,y)\,d\pi(x,y) ,
$$
such that $\norm{\mu}=\norm{m}$ and the value of \eqref{eq:weighted finite} is $\norm{\pi}<\infty$, while \eqref{eq:marginals_finite_first_moment} is the finite first moment of $\lambda^+$ , that is
$$
\int_{\wt{M}} \frac{d(x,0)}{d(x,y)}\,d\mu(x,y) = \int_{\wt{M}} d(x,0)\,d\pi(x,y) = \int_M d(x,0)\,d((\pp_1)_\sharp\pi)(x) = \int_M d(x,0)\,d\lambda^+(x)<\infty .
$$
Moreover $\dual{\Phi}\mu\in\lipfree{M}$ by Proposition \ref{pr:wt_bochner}, and for every $f\in\Lip_0(M)$ we have
\begin{align*}
\duality{\dual{\Phi}\mu,f} = \int_{\wt{M}}(\Phi f)\,d\mu &= \int_{\wt{M}}\frac{f(x)-f(y)}{d(x,y)}\,d\mu(x,y) \\
&= \int_{M\times M}(f(x)-f(y))\,d\pi(x,y) \\
&= \int_{M\times M}(f\circ\pp_1)\,d\pi - \int_{M\times M}(f\circ\pp_2)\,d\pi \\
&= \int_M f\,d\lambda^+ - \int_M f\,d\lambda^- = \int_M f\,d\lambda
\end{align*}
thus $\dual{\Phi}\mu=\widehat{\lambda}=m$. Note again that splitting the integral is valid because the two terms $\int_Mf\,d\lambda^\pm$ are integrable by \cite[Proposition 4.4]{AP_measures}. This finishes the proof of the equivalence (i)$\Leftrightarrow$(ii).

For the last statement, observe that $\mu\ll\pi$ and hence $(\pp_i)_\sharp\mu \ll (\pp_i)_\sharp\pi$ for $i=1,2$. Thus, the fact that $(\pp_1)_\sharp\pi=\lambda^+$ and $(\pp_2)_\sharp\pi=\lambda^-$ are mutually singular implies that $(\pp_1)_\sharp\mu \perp (\pp_2)_\sharp\mu$ as well.
\end{proof}

Note that $m$ can also admit optimal representations $\mu\in\opr{\wt{M}}$ such that $(\pp_1)_\sharp\mu$ and $(\pp_2)_\sharp\mu$ are not mutually singular. As an example,
consider the metric space $M=\set{0,1,2}\subset\mathbb{R}$ and $m=\delta(1)-2\delta(2)\in\lipfree{M}$; recall that $\delta(x)$ stands for the evaluation functional at $x\in M$. Then $\norm{m}=\duality{m,-d(\cdot,0)}=3$; moreover, $m=\widehat{\lambda}$ where $\lambda=\delta_{1}-2\delta_{2}\in\meas{M}$, and $\mu=\delta_{(0,1)}+2\delta_{(1,2)}\in\meas{\wt{M}}$ is an optimal representation, i.e. $\Phi^*\mu=m$, $\mu\in\opr{\wt{M}}$ and satisfies the finiteness conditions \eqref{eq:weighted finite} and \eqref{eq:marginals_finite_first_moment}. However,
both $(\pp_1)_\sharp\mu$ and $(\pp_2)_\sharp\mu$ have positive mass at $1$. In Example \ref{ex:lebesgue_measure_is_convex_series} we also demonstrate that not all representations from $\opr{\wt{M}}$ of functionals induced by Radon measures have to meet conditions \eqref{eq:weighted finite} and \eqref{eq:marginals_finite_first_moment}.

Measures on $M$ inducing elements of $\lipfree{M}$ are not necessarily Radon. For instance, positive functionals in $\lipfree{M}$ are always induced by a positive measure on $M$ \cite[Corollary 5.8]{AP_measures}, however the measure could be $\sigma$-finite but not finite \cite[Remark 5.5]{AP_measures}. As a close relative, we may consider \emph{majorisable}
functionals in $\lipfree{M}$, i.e. those that can be written as the difference between two positive functionals. These are essentially the ``measure-induced elements'' of $\lipfree{M}$, except that they may be induced by the difference of two $\sigma$-finite positive measures, which is not formally a measure \cite[Theorem 5.9]{AP_measures}. 

Majorisable functionals turn out to be convex integrals of molecules as well. In Theorem \ref{th:majorizable_convex_integrals} we will extend (or rather apply) Theorem \ref{th:major_radon_wt} to provide a more general sufficient condition for an element of $\lipfree{M}$ to be a convex integral of molecules, and we will also identify majorisable functionals with those convex integrals of molecules that satisfy condition \eqref{eq:marginals_finite_first_moment}. Hence, \cite[Remark 5.5]{AP_measures} shows that there exist convex integrals of molecules that admit representations from $\opr{\wt{M}}$ satisfying condition \eqref{eq:marginals_finite_first_moment} but not \eqref{eq:weighted finite}, and the latter distinguishes Radon measures in the class of majorisable functionals.

First, to make this more general sufficient condition clear, let us define certain weighting operators on Lipschitz spaces. Given $n\in\NN$, denote 
$$A_n=\left\{x\in M: 2^{-(n+1)}\leq d(x,0)\leq 2^{n+1}\right\}$$ and define the function $\Pi_n\in\Lip_0(M)$ supported in $A_n$ by
$$
\Pi_n(x) = \begin{cases}
0 &\text{if } d(x,0)\leq 2^{-(n+1)} , \\
2^{n+1}d(x,0)-1 &\text{if } 2^{-(n+1)}\leq d(x,0)\leq 2^{-n} , \\
1 &\text{if } 2^{-n}\leq d(x,0)\leq 2^n , \\
2-2^{-n}d(x,0) &\text{if } 2^n\leq d(x,0)\leq 2^{n+1} , \\
0 &\text{if } 2^{n+1}\leq d(x,0).
\end{cases}
$$
Next, let $W_{\Pi_n}:\Lip_0(M)\to\Lip_0(M)$ be the weighting operator given by $W_{\Pi_n}(f)=f\cdot {\Pi_n}$ for every $f\in\Lip_0(M)$ (see \cite[Lemma 2.3]{APPP}). The operator $W_{\Pi_n}$ is $w^*$-$w^*$-continuous and, by \cite[Lemma 4]{AP_normality}, for every $m\in \lipfree{M}$ the $W_{\Pi_n}^*(m)$ converge to $m$ in the norm topology of $\lipfree{M}$. This operator can be viewed as a ``Lipschitz-regular restriction'' of $m$ to the annulus $A_n$. We will say that $m\in\lipfree{M}$ is \emph{majorisable on annuli} if $W_{\Pi_n}^*(m)$ is majorisable (hence measure-induced) for every $n\in\NN$. We will make use of the following characterisation of functionals that are majorisable on annuli.

\begin{lemma}
\label{lem:majorisable_on_annuli}
Let $m\in\lipfree{M}$. Then $m$ is majorisable on annuli if and only if there exist positive $\sigma$-finite Borel measures $\lambda^+$, $\lambda^-$ on $M$ such that:
\begin{enumerate}[label={\upshape{(\alph*)}}]
\item $\lambda^+(\set{0})=\lambda^-(\set{0})=0$,
\item $\lambda^+\restrict_{A_n}, \lambda^-\restrict_{A_n}\in\meas{A_n}$ for every $n\in\NN$,
\item $\lambda^+\perp\lambda^-$,
\item $W_{\Pi_n}^*(m)=\widehat{\lambda_n}$ where $\lambda_n\in\meas{M}$ is given by $d\lambda_n=\Pi_nd\lambda^+ - \Pi_nd\lambda^-$ for every $n\in\NN$.
\end{enumerate}
\end{lemma}

\begin{proof}
The ``if'' part of the statement follows from (d) and \cite[Theorem 5.4]{AP_measures}. For the converse, put $m_n=W_{\Pi_n}^*(m)$ for every $n\in\NN$. By \cite[Theorem 5.4]{AP_measures} there exist Radon measures $\lambda_n\in\meas{M}$ inducing $m_n$ and such that $\abs{\lambda_n}(\set{0})=0$. For any $k> n$ we have $\Pi_n=\Pi_n\Pi_k$, thus $m_n=W_{\Pi_n}^*(m_k)$. Since inducing measures are unique modulo $\delta_0$ by \cite[Proposition 4.9]{AP_measures}, we conclude that $d\lambda_n = \Pi_n\,d\lambda_k$ for $k> n$. In particular, $\lambda_n^+\leq\lambda_{n+1}^+$ and $\lambda_n^-\leq\lambda_{n+1}^-$. So, for Borel $E\subset M$ we can define
$$
\lambda^+(E) = \lim_n\lambda_n^+(E) = \sup_n\lambda_n^+(E)
$$
as a value in $[0,\infty]$. It is straightforward to check that $\lambda^+$ is a positive $\sigma$-finite Borel measure on $M$, and that $d\lambda_n^+=\Pi_n\,d\lambda^+$. The measure $\lambda^-$ is constructed analogously. Properties (a) and (d) are then obvious. Property (b) follows from the fact that for any Borel $A\subset A_n$ we have $\lambda^+(A)=\lambda^+_{n+1}(A)$. To prove (c), let $B_n^+$ and $B_n^-$ be disjoint Borel subsets of $M$ on which $\lambda_n^+$ and $\lambda_n^-$ are concentrated, respectively. Since $\lambda_n^+\ll\lambda_k^+$ for $k\geq n$, $\lambda_n^+$ is concentrated on $\bigcap_{k=n}^\infty B_k^+$ and therefore $\lambda^+$ is concentrated on the Borel set $B^+=\bigcup_{n=1}^\infty\bigcap_{k=n}^\infty B_k^+$. Similarly, $\lambda^-$ is concentrated on $B^-=\bigcup_{n=1}^\infty\bigcap_{k=n}^\infty B_k^-$, and it is clear that $B^+$ and $B^-$ are disjoint.
\end{proof}

Now we are ready to describe more convex integrals of molecules.

\begin{theorem}
\label{th:majorizable_convex_integrals}
Let $m\in\lipfree{M}$. If $m$ is majorisable on annuli, then it is a convex integral of molecules. 

Moreover, $m$ is majorisable (globally on $M$) if and only if it is a convex integral of molecules with a representation $\mu\in\opr{\wt{M}}$ satisfying \eqref{eq:marginals_finite_first_moment}, that is
$$
\int_{\widetilde{M}}\frac{d(x,0)}{d(x,y)}\,d\mu(x,y)<\infty .
$$
\end{theorem}

\begin{proof}
For the first part of the theorem, assume that $m\in\lipfree{M}$ is majorisable on annuli.
For every $n\in\NN$, denote $m_n=W_{\Pi_n}^*(m)$, and let $\lambda^+,\lambda^-,\lambda_n$ be the measures on $M$ constructed in Lemma \ref{lem:majorisable_on_annuli}. By Theorem \ref{th:major_radon_wt}, there exist measures $\mu_n\in\opr{\wt{M}}$ representing $m_n$. Moreover, details of the proof of Theorem \ref{th:major_radon_wt} reveal that $\mu_n$ are given by $d\mu_n(x,y)=d(x,y)\,d\pi_n(x,y)$ for some $\pi_n\in\meas{\wt{M}}$ satisfying $(\pp_1)_\sharp\pi_n=\nu_n^+$ and $(\pp_2)_\sharp\pi_n=\nu_n^-$, where $\nu_n$ is again a measure inducing $m_n$ defined as $\nu_n=\lambda_n-\lambda_n(M)\cdot\delta_0\in\meas{M}$.

Thanks to $M$ being a complete metric space, we may consider measures $\mu_n$ and $(\pp_i)_\sharp(\mu_n)$, $\lambda_n$, $\nu_n$ as Radon measures on $\bwt{M}$ and $\beta M$ that are concentrated on $\wt{M}$ and $M$, respectively. Similarly, the measures $\lambda^+$ and $\lambda^-$ can be extended to positive $\sigma$-finite Borel measures on $\beta M$ that are concentrated on $M$ and have the property that $\lambda^\pm\restrict_{\overline{A_n}^{\beta M}}$ are Radon. The inequalities $\lambda_n^{+}\leq \lambda^{+}$ and $\lambda_n^{-}\leq \lambda^{-}$ for any $n\in \NN$ following from the construction in Lemma \ref{lem:majorisable_on_annuli} will be preserved on $\beta M$.

Because $\norm{\mu_n}=\norm{m_n}$ and the sequence $(m_n)$ is norm convergent, $(\mu_n)$ is bounded in $\meas{\beta\wt{M}}$ and we may find a subnet $(\mu_{n_i})$ and a measure $\mu\in\meas{\beta\wt{M}}$ such that $\int_{\beta\wt{M}}\varphi\,d\mu=\lim_i\int_{\beta\wt{M}}\varphi\,d\mu_{n_i}$ for every $\varphi\in C(\bwt{M})$. We will show that $m$ is a convex integral of molecules with representation $\mu$. Indeed, for every $f\in\Lip_0(M)$ we have
$$
\duality{m,f}=\lim_i\duality{m_{n_i},f}=\lim_i\int_{\beta\wt{M}}(\Phi f)\,d\mu_{n_i}
=\int_{\beta\wt{M}}(\Phi f)\,d\mu=\duality{f,\Phi^*\mu},
$$
so $m=\Phi^*\mu$.
Moreover, as a weak$^*$ limit of positive measures, $\mu$ is also positive. It follows that
$$\|\mu\|=\int_{\bwt{M}}\boldsymbol{1}\,d\mu=\lim_i\int_{\bwt{M}}\boldsymbol{1}\,d\mu_{n_i}=\lim_i\|\mu_{n_i}\|=\lim_i\|m_{n_i}\|=\|m\|$$
and $\mu\in\opr{\bwt{M}}$ is an optimal representation of $m$.

It remains to be shown that $\mu$ is concentrated on $\wt{M}$. To this end we first observe a sort of absolute continuity of the marginals of measure $\mu$ with respect to $\lambda^{\pm}$.

\begin{claim}
If $K$ is a closed subset of $\beta M$ that does not contain $0$ and $\lambda^+(K)=0$, then $(\pp_1)_\sharp\mu(K)=0$. The same is true also for $\lambda^-$ and $(\pp_2)_\sharp\mu$.
\end{claim}

\begin{proof}[Proof of the Claim]
Fix $\varepsilon>0$ and, for $n\in\NN$, denote
\begin{align*}
\mathcal{B}_n &= \set{\xi\in\beta M : d(\xi,0)\leq 2^n} & \mathcal{B}'_n &= \set{\xi\in\beta M : 2^{-n}\leq d(\xi,0)\leq 2^n} \\
\mathcal{U}_n &= \set{\xi\in\beta M : d(\xi,0)< 2^n} & \mathcal{U}'_n &= \set{\xi\in\beta M : 2^{-n}< d(\xi,0)< 2^n}
\end{align*}
where $d(\xi,0)$ is as in \eqref{eq:Mr_definition}.
Observe that $\mu$ is concentrated on $\pp^{-1}(\rcomp{M}\times\rcomp{M})$ by Proposition \ref{pr:finite_concentration}, and 
$\rcomp{M}=\bigcup_{n=1}^\infty \mathcal{B}_n$. So there exists $n\in\NN$ such that
\begin{align*}
(\pp_1)_\sharp\mu(K) = \mu\left(\pp_1^{-1}(K)\right) &\leq \varepsilon + \mu\left(\pp_1^{-1}(K) \cap \pp^{-1}(\mathcal{B}_n\times \mathcal{B}_n)\right) \\
&= \varepsilon + \mu\left(\pp^{-1}((K\cap \mathcal{B}_n)\times \mathcal{B}_n)\right) .
\end{align*}
Because $0\notin K$, we may choose $n$ large enough that $K\cap\mathcal{B}_n\subset\mathcal{B}'_n$. Since $\lambda^+(K\cap\mathcal{B}_n)=0$ and $\lambda^+$ is Radon on $\mathcal{B}'_{n+1}\subset \overline{A_{n+1}}^{\beta M}$, we may find an open set $V$ in $\beta M$ such that $K\cap\mathcal{B}_n\subset V\subset\mathcal{U}'_{n+1}$ and $\lambda^+(V)\leq 2^{-(n+2)}\varepsilon$.
Define a function $\varphi\in C(\bwt{M})$ so that $0\leq \varphi\leq 1$, $\varphi=1$ on the compact set $\pp^{-1}((K\cap \mathcal{B}_n) \times \mathcal{B}_n)$ and $\varphi=0$ on the compact set $\bwt{M}\setminus \pp^{-1}(V\times \mathcal{U}_{n+1})$. Then we get
\begin{align*}
\mu\left(\pp^{-1}((K\cap \mathcal{B}_n)\times \mathcal{B}_n)\right) &\leq  \int_{\bwt{M}}\varphi\,d\mu = \lim_i\int_{\bwt{M}}\varphi\,d\mu_{n_i} \leq \lim_i \mu_{n_i}\left(\pp^{-1}(V\times \mathcal{U}_{n+1})\right) \\
&= \lim_i \int_{\pp^{-1}(V\times \mathcal{U}_{n+1})\cap\wt{M}}d(x,y)\,d\pi_{n_i}(x,y) \\
&\leq 2^{n+2}\lim_i \int_{\pp^{-1}(V\times \mathcal{U}_{n+1})\cap\wt{M}}\,d\pi_{n_i} \\
&\leq 2^{n+2}\lim_i \pi_{n_i}\left(\pp_1^{-1}(V)\right) \\
&= 2^{n+2}\lim_i\nu_{n_i}^+(V) \\
&= 2^{n+2}\lim_i\lambda_{n_i}^+(V) \leq 2^{n+2}\lambda^+(V) \leq \varepsilon
\end{align*}
where the equality $\nu_{n_i}^+(V)=\lambda_{n_i}^+(V)$ holds because $0\notin V$. We conclude that $(\pp_1)_\sharp\mu(K)\leq 2\varepsilon$, and letting $\varepsilon$ tend to 0 yields the first part of the Claim. A similar argument proves the second part.
\end{proof}

Let us now proceed to prove that $\mu(\bwt{M}\setminus \wt{M})=0$. We write 
$$\bwt{M}\setminus \wt{M}=\pp_1^{-1}(\beta M\setminus M)\cup \pp_2^{-1}(\beta M\setminus M)\cup \left\{\xi\in\bwt{M}: \pp_1(\xi)=\pp_2(\xi)\in M\right\}$$ and treat these subsets separately. 

For every $l\in\NN$ define
$$
\mathcal{K}_l=\set{\xi\in\beta{M}:\inf_{x\in M}d(\xi,x)\geq \frac{1}{l}}
$$
where, again, $d(\xi,x)$ stands for the evaluation at $\xi$ of the continuous extension of $d(\cdot,x)$ to $\beta M$.
Clearly, $\mathcal{K}_l$ is a subset of $\beta M\setminus M$ closed in $\beta M$. Since $\lambda^+$ and $\lambda^-$ are concentrated on $M$, $\lambda^+(\mathcal{K}_l)=0=\lambda^-(\mathcal{K}_l)$ and the Claim above implies that $\mu(\pp_1^{-1}(\mathcal{K}_l))=0=\mu(\pp_2^{-1}(\mathcal{K}_l))$. It is shown in \cite{Weaver_old} as a part of the proof of Proposition 2.1.6 that $\beta M\setminus M=\bigcup_{l\in\NN}\mathcal{K}_l$ (see also \cite[Lemma 4.10 and Theorem 4.11]{AP_measures} and references therein). So, we can conclude that $\mu(\pp_1^{-1}(\beta M\setminus M))=0=\mu(\pp_2^{-1}(\beta M\setminus M))$.

In the last step, we deal with the ``diagonal'' in $\bwt{M}$, the set $\mathcal{X}=\{\xi\in\bwt{M}: \pp_1(\xi)=\pp_2(\xi)\in M\}$. We partition it as
$$\mathcal{X}=\pp^{-1}(0,0)\cup\left(\mathcal{X}\cap \pp_2^{-1}(B^+\setminus\{0\})\right)\cup\left(\mathcal{X}\cap \pp_1^{-1}(\beta M\setminus (B^+\cup\{0\}))\right)\,,$$
where $B^+$ and $B^-$ are disjoint Borel subsets of $M$ on which the measures $\lambda^+$ and $\lambda^-$ are concentrated, respectively. Let $\mathcal{Y}$ be an arbitrary compact subset of $\mathcal{X}\cap \pp_2^{-1}(B^+\setminus\{0\})$. Then $\pp_2(\mathcal{Y})$ is a compact subset of $B^+\setminus\{0\}$ and $\lambda^-(\pp_2(\mathcal{Y}))=0$. So, the Claim yields 
$$\mu(\mathcal{Y})\leq \mu\left(\pp_2^{-1}(\pp_2(\mathcal{Y}))\right)=(\pp_2)_\sharp\mu(\pp_2(\mathcal{Y}))=0$$ and by the regularity of the measure $\mu$ we infer that $\mu\left(\mathcal{X}\cap \pp_2^{-1}(B^+\setminus\{0\})\right)=0$. Similarly, if we take any compact subset $\mathcal{Z}$ of $\mathcal{X}\cap \pp_1^{-1}(\beta M\setminus (B^+\cup\{0\}))$, then $\pp_1(\mathcal{Z})$ is a compact subset of $\beta M\setminus (B^+\cup\{0\})$ and $\lambda^+(\pp_1(\mathcal{Z}))=0$. Applying the Claim again gives
$$\mu(\mathcal{Z})\leq \mu\left(\pp_1^{-1}(\pp_1(\mathcal{Z}))\right)=(\pp_1)_\sharp\mu(\pp_1(\mathcal{Z}))=0$$ and therefore also $\mu\left(\mathcal{X}\cap \pp_1^{-1}(\beta M\setminus (B^+\cup\{0\}))\right)=0$. Finally, $\mu\left(\pp^{-1}(0,0)\right)=0$ by Lemma \ref{lm:measure_over_a_point}, and we obtain that $\mu(\mathcal{X})=0$ as desired. This completes the proof that $\mu\in\opr{\wt{M}}$ and thus the proof of the first assertion of the theorem.

For the second statement of the theorem, assume now that $m$ itself is majorisable. Then clearly it is majorisable on annuli and hence a convex integral of molecules by the first part of the proof. Moreover, by \cite[Theorem 5.16]{AP_measures}, the measures $\lambda^{\pm}$ in Lemma \ref{lem:majorisable_on_annuli} applied to a majorisable $m$ can be chosen to have a finite first moment. We will verify that using such measures in the construction of the representation $\mu$ above yields condition \eqref{eq:marginals_finite_first_moment}. Denote
$$\psi:\wt{M}\to\RR:(x,y)\mapsto\frac{d(x,0)}{d(x,y)}.$$
Take a sequence of non-negative functions $(\psi_k)\subset C_b(\wt{M})$ increasing pointwise to $\psi$, e.g. $\psi_k=\min\set{\psi,k}$.
Then for each $k,n\in\NN$ we have
\begin{align*}
\int_{\wt{M}}\psi_k\,d\mu_{n}&\leq \int_{\wt{M}}\psi\,d\mu_{n}=\int_{\wt{M}}d(x,0)\,d\pi_{n}(x,y)=\int_{M}d(x,0)\,d((\pp_1)_{\sharp}\pi_{n})(x)\\
&=\int_{M}d(x,0)\,d\nu^+_n(x)
=\int_{M}d(x,0)\,d\lambda^+_n(x)\\
&\leq \int_{M}d(x,0)\,d\lambda^+(x)<\infty
\end{align*}
because $\lambda^+$ has a finite first moment. Considering the net $(\mu_{n_i})$ converging weak$^*$ to $\mu$ that we selected earlier, we obtain for every $k\in\NN$ that
$$\int_{\wt{M}}\psi_k\,d\mu=\lim_i\int_{\wt{M}}\psi_k\,d\mu_{n_i}\leq \int_{M}d(x,0)\,d\lambda^+(x)<\infty$$
as all the measures $\mu_n$ and $\mu$ are concentrated on $\wt{M}$. Therefore, by the monotone convergence theorem,
$$\int_{\wt{M}}\psi\,d\mu=\int_{\wt{M}}\frac{d(x,0)}{d(x,y)}\,d\mu(x,y)<\infty,$$
which is exactly the desired condition \eqref{eq:marginals_finite_first_moment}.

Finally, let $m$ be a convex integral of molecules with a representation $\mu\in\opr{\wt{M}}$ satisfying \eqref{eq:marginals_finite_first_moment}. If we consider the positive Borel measure $\omega$ on $\wt{M}$ given by
$$
d\omega=d(x,y)^{-1}d\left(\mu\restrict_{\pp_1^{-1}(M\setminus\{0\})}\right),
$$
then for any $n\in\NN$,
$$
\omega(\pp_1^{-1}(A_n)) \leq 2^{n+1}\int_{\pp_1^{-1}(A_n)}\frac{d(x,0)}{d(x,y)}\,d\mu(x,y) < \infty
$$
by \eqref{eq:marginals_finite_first_moment}, so $\omega$ is $\sigma$-finite. If we next define
$\lambda=(\pp_1)_{\sharp}\omega$, it follows that $\lambda$ is a positive Borel measure on $M$ satisfying $\lambda(\set{0})=0$. The measure $\lambda$ is also $\sigma$-finite, because $\lambda(A_n)<\infty$ for every $n\in\NN$, and \emph{inner regular}  (i.e. $\lambda(A)=\sup\{\lambda(K): K\subset A\textup{ compact}\}$ for every Borel set $A\subset M$). The latter can be seen by writing any Borel set $A\subset M$ as $A=(A\cap\{0\})\cup\bigcup_{n=1}^{\infty}(A\cap A_n)$ and using the fact that $\omega$ is inner regular on the sets $\pp_1^{-1}(A_n)$.
Moreover, hypothesis \eqref{eq:marginals_finite_first_moment} implies that $\lambda$ has a finite first moment.
Therefore $\lambda$ induces a positive functional $\widehat{\lambda}\in\lipfree{M}$ as shown in \cite[Proposition 4.4 and Proposition 4.7]{AP_measures}. Because for every $f\in\Lip_0(M)$ satisfying $f\geq 0$ pointwise we have
\begin{align*}
\duality{m,f} &= \int_{\wt{M}}\Phi f\,d\mu \leq \int_{\wt{M}}\frac{f(x)}{d(x,y)}\,d\mu(x,y) = \int_{\wt{M}}\frac{f(x)}{d(x,y)}\,d\left(\mu\restrict_{\pp_1^{-1}(M\setminus\{0\})}\right)(x,y)\\
&= \int_{\wt{M}}f(\pp_1(x,y))\,d\omega(x,y) = \int_{M}f(x)\,d((\pp_1)_\sharp\omega)(x) = \int_{M}f\,d\lambda = \duality{\widehat{\lambda},f},
\end{align*}
$\widehat{\lambda}-m$ is also a positive functional in $\lipfree{M}$ and $m=\widehat{\lambda}-(\widehat{\lambda}-m)$ is majorisable. This concludes the proof of the theorem.
\end{proof}

Examples of convex integrals of molecules that are majorisable on annuli but not majorisable, or not even majorisable on annuli, can be obtained from \cite[Example 4.17]{AP_measures} and \cite[Example 3.24]{Weaver2}, respectively. Here we combine both to present such examples in the familiar setting of the metric space $M=[0,\infty)$.

\begin{example}[{cf. \cite[Example 4.17]{AP_measures} and \cite[Example 3.24]{Weaver2}}]\label{ex:not_majorisable}
Let $M=[0,\infty)$ with the usual metric. Pick numbers $a_n>0$, $n \in \NN$, so that $\sum_{n=1}^\infty a_n < \infty$, and select sequences $(x_n)_{n=1}^\infty, (y_n)_{n=1}^\infty\subset (0,\infty)$ so that all points are distinct, their union is a discrete space and $x_n>y_n$ for every $n\in\NN$. Define $\mu=\sum_{n=1}^\infty a_n\delta_{(x_n,y_n)} \in \meas{\wt{M}}$, $m=\Phi^*\mu=\sum_{n=1}^\infty a_n m_{x_ny_n} \in \lipfree{M}$ and $f \in \Lip_0(M)$ by $f(x)=x$. We see that $m$ is a convex integral (moreover a convex series) of molecules by applying Lemma \ref{lm:norming_phi1} to $f$.

Let us suppose that we can write $m=m^+-m^-$, where $m^\pm \in \lipfree{M}$ are positive.
Pick $f_n \in \Lip_0(M)$, $n \in \NN$, satisfying $0 \leq f_n \leq f$, $f_n(x_k)=x_k$ for $k=1,\ldots,n$, $f_n(x_k)=0$ for $k>n$, and $f_n(y_k)=0$ for all $k\in\NN$. Then
\[
\duality{m^+,f} \geq \duality{m^+,f_n} \geq \duality{m,f_n} = \sum_{k=1}^n \frac{a_k x_k}{x_k-y_k}, \quad n \in \NN.
\]
Thus if we set $a_n=2^{-n}$, $x_n=2^n$ and $y_n=x_n - 1$, $n \in \NN$, we get a contradiction and therefore $m$ is not majorisable. So, no optimal representation of $m$ concentrated on $\wt{M}$ can satisfy condition \eqref{eq:marginals_finite_first_moment}. The functional $m$ is however majorisable on annuli because the intersection of its support with each annulus is finite, and the representation $\mu\in\opr{\wt{M}}$ above meets condition \eqref{eq:weighted finite}.

On the other hand, if we set $a_n=2^{-n}$, $x_n=\frac{1}{2}+2^{-n}$ and $y_n=x_n-5^{-n}$, then $m$ is not majorisable on annuli because it is not majorisable, and its support is included in the first annulus $A_1$. Also, all optimal representations $\nu$ of $m$ concentrated on $\wt{M}$ fail both conditions \eqref{eq:weighted finite} and \eqref{eq:marginals_finite_first_moment}. Condition \eqref{eq:marginals_finite_first_moment} cannot hold because $m$ is not majorisable, and \eqref{eq:weighted finite} would in this case imply \eqref{eq:marginals_finite_first_moment}. Indeed, consider $g \in \norming{m}$ given by $g(x)=\max\set{1-|1-x|,0}$. By Lemma \ref{lm:norming_phi1}, given $(x,y) \in \supp(\nu) \cap \wt{M}$, $1=\Phi g(x,y) \leq g(x)/d(x,y)$ implies $d(x,0) \leq 2$. As $\nu$ is concentrated on $\supp(\nu) \cap \wt{M}$, if \eqref{eq:weighted finite} held then so would \eqref{eq:marginals_finite_first_moment}.
\end{example}

Next we present a condition for $\lipfree{M}$ which ensures that all of its elements are convex integrals of molecules. This condition is neither sufficient nor necessary for the majorisability on annuli of the elements of $\lipfree{M}$ by \cite[Theorem 6.2]{AP_measures} and \cite[Theorem 5]{DKP_2016}. It holds, for instance, when $M$ is a Lebesgue-null closed subset of $\RR$. So, in particular, it generalises the examples constructed above.

\begin{proposition}
\label{pr:isometric_l1}
If $\lipfree{M}$ is isometric to $\ell_1(\Gamma)$ for some index set $\Gamma$, then all elements of $\lipfree{M}$ are convex series of molecules.
\end{proposition}

\begin{proof}
According to \cite[Theorem 5]{DKP_2016}, $\lipfree{M}$ is isometric to $\ell_1(\Gamma)$ precisely when $M$ is a subset of an $\RR$-tree that contains all branching points and has null length measure (see \cite{DKP_2016} for the definitions of these notions). Fix $m\in\lipfree{M}$, then $S=\supp(m)\cup\set{0}$ is separable and thus contained in a separable $\RR$-tree. Hence the set $A=\cl{S\cup\Br(S)}$ (where $\Br(S)$ denotes the set of branching points of $S$) is separable, and contained in $M$ because $\Br(S)\subset\Br(M)$. So by the same theorem, $\lipfree{A}$ is isometric to $\ell_1$. This implication was in fact proved by Godard in \cite[Corollary 3.4]{Godard_2010}, where he also gave an explicit description of the isometry $T:\ell_1\to\lipfree{A}$, and it is easy to verify that it maps elements of the standard basis $(e_n)$ to molecules, say $Te_n=m_{x_ny_n}$ (where $[x_n,y_n]$ are precisely the ``gaps'' in $A$ of positive length). Then we may write
$$
m = T(T^{-1}m) = T\pare{\sum_{n=1}^\infty\duality{T^{-1}m,\dual{e}_n}e_n} = \sum_{n=1}^\infty a_n m_{x_ny_n}
$$
where $(\dual{e}_n)\subset\dual{\ell}_1$ are the coordinate functionals and the coefficients $a_n=\duality{T^{-1}m,\dual{e}_n}$ satisfy $\sum\abs{a_n}=\norm{T^{-1}m}_1=\norm{m}$.
\end{proof}

\subsection{Consequences for convex series of molecules}

We will now derive several consequences of Theorems \ref{th:major_radon_wt} and \ref{th:majorizable_convex_integrals}. The first one follows from the trivial observation: if every element of $\lipfree{M}$ is majorisable, then it will also be a convex integral of molecules. Metric spaces where the former condition holds were identified in \cite[Section 6]{AP_measures} and given the name \emph{radially discrete}. They are precisely those spaces $M$ where there exists $\alpha>0$ such that $d(x,y)\geq\alpha\cdot d(x,0)$ for all $x\neq y\in M$. This condition implies that every point other than $0$ is isolated, so radially discrete spaces are scattered. Thus, combining \cite[Theorem 6.2]{AP_measures}, Theorem \ref{th:majorizable_convex_integrals} and Proposition \ref{pr:scattered_series} yields:

\begin{corollary}
\label{cr:rd_cs}
If $M$ is radially discrete then every element of $\lipfree{M}$ is a convex series of molecules.
\end{corollary}

When $M$ is radially discrete and the base point of $M$ is isolated, then $M$ is actually \emph{uniformly discrete}, meaning that $\inf\set{d(x,y)\,:\,x\neq y\in M}>0$. Such spaces were called \emph{radially uniformly discrete} (RUD) in \cite{AP_measures}.
Despite the formulation, being RUD does not depend on the choice of base point.
RUD spaces include, but are not limited to, all bounded and uniformly discrete spaces. Moreover, by \cite[Corollary 6.3]{AP_measures} they are exactly those $M$ such that every element of $\lipfree{M}$ is induced by a Radon measure. Thus $\lipfree{M}$ consists entirely of convex series of molecules satisfying both finiteness conditions \eqref{eq:weighted finite} and \eqref{eq:marginals_finite_first_moment} when $M$ is RUD. 

If $M$ is uniformly discrete, then any element $m\in\lipfree{M}$ with bounded support is majorisable. This follows by considering $m$ as an element of $\lipfree{\supp(m)\cup\{0\}}$. Indeed, bounded subsets of such $M$ are RUD spaces, so by \cite[Corollary 6.3]{AP_measures} $m$ is majorisable in $\lipfree{\supp(m)\cup\{0\}}$, but any positive element thereof is a positive element of $\lipfree{M}$. Since $W_{\Pi_n}^*(m)$ have bounded support for every $m\in\lipfree{M}$ and $n\in\NN$, all elements of $\lipfree{M}$ are majorisable on annuli when $M$ is uniformly discrete. Hence, applying Theorem \ref{th:majorizable_convex_integrals} and Proposition \ref{pr:scattered_series} again, we obtain:

\begin{corollary}
\label{cr:ud_cs}
If $M$ is uniformly discrete then every element of $\lipfree{M}$ is a convex series of molecules.
\end{corollary}

\bigskip

A natural condition to consider for a convex series of molecules
$$
m=\sum_{n=1}^\infty a_n m_{x_ny_n}
$$
on any metric space $M$ is that the sets $\set{x_n}$ and $\set{y_n}$ of first and second coordinates be disjoint. Suppose that this does not hold so that, for instance, the terms $a m_{xy} + b m_{yz}$ appear in the sum. Then the corresponding de Leeuw representation has positive mass at $(x,y)$ and $(y,z)$, and Corollary \ref{cr:metric_triples_pm} implies that $y\in [x,z]$. In that case
\begin{equation}
\label{eq:molecule_splitting}
m_{xz} = \frac{d(x,y)}{d(x,z)}m_{xy} + \frac{d(y,z)}{d(x,z)}m_{yz}
\end{equation}
is a convex combination of $m_{xy}$ and $m_{yz}$, and it is easy to check that the terms $am_{xy}+bm_{yz}$ can then be replaced by either $a' m_{xz} + b' m_{yz}$ or $a' m_{xy} + b' m_{xz}$ for some coefficients $a',b'$ such that $a'+b'=a+b$. Thus it is always possible to remove a single instance of overlap between first and second coordinates, or finitely many of them. But we do not know whether this is possible in general when there may be infinitely many overlapping molecules.

\begin{question}
Suppose that $m\in\lipfree{M}$ can be written as a convex series of molecules $\sum_n a_nm_{x_ny_n}$. Is it always possible to choose an expression where $x_i\neq y_j$ for all $i,j$?
\end{question}

We are only able to provide a positive answer for general $M$ if the series of molecules satisfies the finiteness conditions \eqref{eq:weighted finite} and \eqref{eq:marginals_finite_first_moment}, which in case of series of molecules read as \eqref{eq:convex_series_finite_sum} and \eqref{eq:convex_series_finite_first_moment} below. This follows from observing that the condition $(\pp_1)_\sharp\mu \perp (\pp_2)_\sharp\mu$ in Theorem \ref{th:major_radon_wt} is a continuous counterpart to the non-overlapping of first and second coordinates in the discrete case. 

\begin{corollary}
\label{cr:cs_non_overlapping}
Suppose that $m\in\lipfree{M}$ can be written as a convex series of molecules $m=\sum_{n=1}^\infty a_nm_{x_ny_n}$ such that
\begin{equation}
\label{eq:convex_series_finite_sum}
\sum_{n=1}^\infty \frac{a_n}{d(x_n,y_n)} < \infty
\end{equation}
and
\begin{equation}
\label{eq:convex_series_finite_first_moment}
\sum_{n=1}^\infty a_n\frac{d(x_n,0)}{d(x_n,y_n)} < \infty .
\end{equation}
Then it can also be written as a convex series of molecules $m=\sum_{n=1}^\infty b_nm_{p_nq_n}$ such that
\begin{enumerate}[label={\upshape{(\alph*)}}]
\item $\sum_n\frac{b_n}{d(p_n,q_n)}<\infty$ and $\sum_nb_n\frac{d(p_n,0)}{d(p_n,q_n)}<\infty$,
\item $\set{p_n:n\in\NN}\subset\set{x_n:n\in\NN}$,
\item $\set{q_n:n\in\NN}\subset\set{y_n:n\in\NN}$, and
\item the sets $\set{p_n:n\in\NN}$ and $\set{q_n:n\in\NN}$ are disjoint.
\end{enumerate}
\end{corollary}

Note that conditions \eqref{eq:convex_series_finite_sum} and \eqref{eq:convex_series_finite_first_moment} are automatically satisfied when $M$ is uniformly discrete and radially discrete, respectively. Thus if $M$ is RUD, in particular if it is uniformly discrete and bounded, then every element of $\lipfree{M}$ can be written as a convex series of molecules with non-overlapping first and second coordinates.

\begin{proof}
Condition \eqref{eq:convex_series_finite_sum} implies that the sums
$$
\nu_1 = \sum_{n=1}^\infty \frac{a_n}{d(x_n,y_n)}\,\delta_{x_n} \qquad\text{and}\qquad \nu_2 = \sum_{n=1}^\infty \frac{a_n}{d(x_n,y_n)}\,\delta_{y_n}
$$
are absolutely convergent and describe discrete positive Radon measures on $M$, and condition \eqref{eq:convex_series_finite_first_moment} in turn ensures that these measures have finite first moment. It is clear then that $m$ is induced by the Radon measure $\lambda=\nu_1-\nu_2$.
The Jordan decomposition $\lambda=\lambda^+-\lambda^-$ is thus such that $\lambda^+\leq\nu_1$ is concentrated on the set $\set{x_n}$, and similarly $\lambda^-$ is concentrated on $\set{y_n}$. Let $\mu$ be the representation of $m$ given by Theorem \ref{th:major_radon_wt}. Then $(\pp_1)_\sharp\mu\ll\lambda^+$ and $(\pp_2)_\sharp\mu\ll\lambda^-$, hence both $(\pp_1)_\sharp\mu$ and $(\pp_2)_\sharp\mu$ are concentrated on a countable set, and so the same holds for $\mu$. Thus
$$
\mu = \sum_{n=1}^\infty b_n\delta_{(p_n,q_n)}
$$
for some $(p_n,q_n)\in\wt{M}$, $b_n>0$ and $\sum_nb_n=\norm{m}$.
Then $\set{p_n}$ and $\set{q_n}$ are the smallest sets where $(\pp_1)_\sharp\mu$ and $(\pp_2)_\sharp\mu$ are concentrated, respectively, and properties (b), (c) and (d) follow immediately. Moreover, Theorem \ref{th:major_radon_wt} also states that
$$
\sum_{n=1}^\infty\frac{b_n}{d(p_n,q_n)} = \int_{\wt{M}}\frac{d\mu(x,y)}{d(x,y)} < \infty
$$
and
$$
\sum_{n=1}^\infty b_n\frac{d(p_n,0)}{d(p_n,q_n)} =\int_{\wt{M}}\frac{d(x,0)}{d(x,y)}\,d\mu(x,y) < \infty,
$$
which yields (a).
\end{proof}

As shown in the proof of Corollary \ref{cr:cs_non_overlapping}, the finiteness hypotheses \eqref{eq:convex_series_finite_sum} and \eqref{eq:convex_series_finite_first_moment} imply that $m$ is induced by a discrete Radon measure on $M$. On the other hand, if an element of $\lipfree{M}$ is induced by a Radon measure on $M$, then the inducing measure is unique modulo $\delta_0$ \cite[Proposition 4.9]{AP_measures}. Thus the argument above shows that an element of $\lipfree{M}$ induced by a non-discrete Radon measure on $M$ can never be written as a convex series of molecules where \eqref{eq:convex_series_finite_sum} and \eqref{eq:convex_series_finite_first_moment} hold.

However, if we allow the sums \eqref{eq:convex_series_finite_sum} and \eqref{eq:convex_series_finite_first_moment} to be infinite then it is quite possible for such induced elements to be written as convex series. We finish this section by characterising those elements of $\lipfree{\RR}$ that can be expressed as convex series, and in so doing show that the element induced by Lebesgue measure on $[0,1]$ can be so written. Note that this also furnishes an example of a convex integral of molecules induced by a Radon measure on $M$ that admits an optimal representation in $\opr{\wt{M}}$ failing both finiteness conditions \eqref{eq:weighted finite} and
 \eqref{eq:marginals_finite_first_moment}.

First we recall some elementary facts about integration. By a step function on $\RR$, we mean a finite linear combination of indicator functions of bounded intervals. Let $L_{\mathrm{inc}}$ denote the set of all functions $f:\RR \to \RR$ that are the limits almost everywhere of a series of non-negative step functions. We know from the approach of F. Riesz to Lebesgue integration that each $f \in L_1$ can be written $f=f_1-f_2$, where $f_1,f_2 \in L_{\mathrm{inc}}$ (see e.g. \cite[Exercise 2.12.60]{Bogachev}). Given $(x,y) \in \wt{\RR}$, define the step function
\[
 \phi_{xy} = \frac{\mathbf{1}_{(\min\{x,y\},\max\{x,y\})}}{x-y}. 
\]
Evidently any non-negative step function $\phi$ can be written as $\phi = \sum_{k=1}^n a_k \phi_{x_k y_k}$ almost everywhere, with $y_k < x_k$, $a_k \geq 0$, $k=1,\ldots,n$, and $\norm{\phi}_1 = \sum_{k=1}^n a_k$.
Now recall the standard isometric isomorphism $T:L_1\to \lipfree{\RR}$, given by
\[
 \duality{Tf,g} = \int_{-\infty}^\infty f(t)g'(t)\,dt, \qquad g \in \Lip_0(\RR),
\]
where by absolute continuity of $g$ we have $g' \in L_\infty$ and $\norm{g'}_\infty=\lipnorm{g}$.  Moreover, for $g \in \Lip_0(\RR)$
\[
 \duality{T\phi_{xy},g} = \frac{1}{x-y}\int_{\min\{x,y\}}^{\max\{x,y\}} g'(t) dt = \frac{g(x)-g(y)}{|x-y|} = \duality{m_{xy},g},
\]
so $T\phi_{xy} = m_{xy}$.

\begin{proposition}\label{pr:convex_series_in_F}
 Let $m \in \lipfree{\RR}$. Then $m$ can be written as a convex series of molecules if and only if $(T^{-1}m)^\pm \in L_{\mathrm{inc}}$.
\end{proposition}

\begin{proof}
 Let $m = \sum_{n=1}^\infty a_n m_{x_n y_n}$, where $(x_n,y_n)\in\wt{\RR}$, $a_n \geq 0$ and $\sum_{n=1}^\infty a_n = \norm{m}$. Let $I=\set{n \in \NN \,:\, x_n > y_n}$ and $J=\NN\setminus I$. Define the positive functions
 \[
  f_1 = \sum_{n \in I} a_n \phi_{x_n y_n} \qquad\text{and}\qquad f_2 = \sum_{n \in J} -a_n \phi_{x_n y_n}.
 \]
By the monotone convergence theorem the corresponding series of non-negative step functions converge almost everywhere to $f_1$ and $f_2$, respectively, and thus $f_1,f_2 \in L_{\mathrm{inc}}$, with $\norm{f_1}_1 = \sum_{n \in I} a_n$ and $\norm{f_2}_1 = \sum_{n \in J} a_n$. Moreover
\[
 T(f_1 - f_2) = \sum_{n \in I} a_n m_{x_n y_n} + \sum_{n \in J} a_n m_{x_n y_n} = \sum_{n=1}^\infty a_n m_{x_n y_n} = m,
\]
so $f:= T^{-1}m = f_1-f_2$ and $\norm{f}_1 = \norm{m} = \norm{f_1}_1 + \norm{f_2}_1$. As $f^+ \leq f_1$ and $f^- \leq f_2$, we obtain
\[
 0 \leq \norm{f_1 - f^+}_1+\norm{f_2 - f^-}_1 = \norm{f_1 - f^+ + f_2 - f^-}_1 \leq \norm{f_1}_1+\norm{f_2}_1-\norm{|f|}_1 = 0,
\]
whence $f^+=f_1$ and $f^-=f_2$ almost everywhere, meaning that both functions belong to $L_{\mathrm{inc}}$.

Conversely, let $f:=T^{-1}m$, with $f^+,f^- \in L_{\mathrm{inc}}$. Given that any non-negative step function $\phi$ can be written almost everywhere as a finite linear combination of the $\phi_{xy}$ as above, we can write 
\[
 f^+ = \sum_{n=1}^\infty a_n \phi_{x_n y_n} \qquad\text{and}\qquad f^- = \sum_{n = 1}^\infty a'_n \phi_{x'_n y'_n}
\]
almost everywhere, with $y_n<x_n$, $y'_n<x'_n$, $a_n,a'_n \geq 0$, $\norm{f^+}_1 = \sum_{n=1}^\infty a_n$ and $\norm{f^-}_1 = \sum_{n=1}^\infty a'_n$. Then
\[
 m = T(f^+-f^-) = \sum_{n=1}^\infty a_n m_{x_n y_n} + \sum_{n=1}^\infty a'_n m_{y'_n x'_n},
\]
and $\norm{m} \leq \sum_{n=1}^\infty a_n + \sum_{n=1}^\infty a'_n = \norm{f^+}_1 + \norm{f^-}_1 = \norm{f}_1 = \norm{m}$, so $m$ is a convex series of molecules.
\end{proof}

\begin{example}
\label{ex:lebesgue_measure_is_convex_series}
The element $m \in \lipfree{\RR}$ induced by Lebesgue measure on $[0,1]$ is a convex series of molecules. Indeed, consider $f \in L_1$ given by $f(x)=1-x$, $x \in [0,1]$, and $f(x)=0$ otherwise. Using integration by parts,
\[
  \duality{Tf,g} = \int_0^1 (1-t)g'(t)\,dt = (1-t)g(t)\bigg|_0^1 + \int_0^1 g(t)\,dt = \duality{m,g}
\]
whenever $g \in \Lip_0(\RR)$. Therefore $m = Tf$. We can verify that $f \in L_{\mathrm{inc}}$, with 
\[
 f = \sum_{n=0}^\infty \sum_{k=0}^{2^n-1} \frac{1}{4^{n+1}} \phi_{\frac{2k+1}{2^{n+1}},\frac{2k}{2^{n+1}}},
\]
hence
\[
 m = Tf = \sum_{n=0}^\infty \sum_{k=0}^{2^n-1} \frac{1}{4^{n+1}} m_{\frac{2k+1}{2^{n+1}},\frac{2k}{2^{n+1}}}.
\]
One can also immediately check that both sums \eqref{eq:convex_series_finite_sum} and \eqref{eq:convex_series_finite_first_moment} are infinite for this particular series.
\end{example}

\section{Functionals not expressible as convex integrals or series of molecules}
\label{sec:negative}

We will now show examples of elements of $\lipfree{M}$ that cannot be expressed as a convex integral or series of molecules. An example is given in \cite[Example 3.2]{ARZ} (although it can be traced back to \cite{KMS} at least) of an element $m\in\lipfree{[0,1]}$ that does not admit any expression as a convex series of molecules, and a straightforward modification of the argument therein shows that it also cannot be a convex integral of molecules. This element is constructed around a \emph{fat Cantor set}, i.e. a compact, nowhere dense subset of $\RR$ with positive Lebesgue measure. The canonical example thereof is the well-known Smith-Volterra-Cantor set, built iteratively in a similar way as the middle-thirds Cantor set but with different ratios so that the result has positive measure. With our formulation, a subset of $\RR$ has positive measure if and only if it contains a fat Cantor set.

The underlying idea behind \cite[Example 3.2]{ARZ} is that $m$ is normed by a function $f\in\norming{m}$, a variant of Cantor's staircase, that does not attain its Lipschitz constant between any pair of different points. By Lemma \ref{lm:norming_phi1}, this forces any optimal representation of $m$ to be supported outside of $\wt{M}$. We now extend that example to any metric space containing an isometric copy of a subset of $\RR$ with positive Lebesgue measure.

\begin{theorem}
\label{th:isometric_fat_cantor_example}
If $M$ contains an isometric copy of a subset of $\RR$ with positive Lebesgue measure, then there is a non-zero $m\in\lipfree{M}$ with the property that $\supp(\mu)\cap\wt{M}=\varnothing$ for any optimal de Leeuw representation $\mu$ of $m$.
\end{theorem}

Let us remark that the extension of the above statement from a fat Cantor set $C$ to an overspace $M$ is not immediately obvious: it is conceivable that there could exist a functional $m\in\lipfree{C}$ that cannot be expressed as a convex integral of molecules in $C$ but can be expressed as a convex integral of molecules in $M$, making use of the additional points in $M$.

\begin{proof}
Denote $I=[0,1]\subset\RR$. Without loss of generality, suppose that there exist a compact set $C\subset M$ containing $0$ and an isometric embedding $\varphi:C\to I$ such that $\varphi(0)=0$, $\varphi(p)=1$ for some $p\in C$, and $\varphi(C)$ is a fat Cantor set.
Then $\varphi(x)=d(x,0)$ for all $x \in C$, and $\varphi$ can be extended to a function in $B_{\Lip_0(M)}$ by setting $\varphi(x)=\min\set{d(x,0),1}$ for all $x\in M$.
Denote $\alpha:=\lambda(\varphi(C))>0$, where $\lambda$ stands for Lebesgue measure.
Let $a_n,b_n\in C$ be such that $I\setminus\varphi(C)$ is the disjoint union of the open intervals $(\varphi(a_n),\varphi(b_n))$ for $n\in\NN$, and define
$$
m := \delta(p) - \sum_{n=1}^\infty (\delta(b_n)-\delta(a_n)) .
$$
The $n$-th term of the series has norm $d(a_n,b_n)=\varphi(b_n)-\varphi(a_n)$ and these norms sum up to ${\lambda(I\setminus\varphi(C))}=1-\alpha$, so the series converges absolutely and $m\in\lipfree{M}$.

Let us start by showing that $\norm{m}=\alpha$. Let $f \in B_{\Lip_0(M)}$. Then $f \circ \varphi^{-1} \in B_{\Lip_0(\varphi(C))}$, so by the McShane extension theorem (see e.g. \cite[Theorem 1.33]{Weaver2}), we can find an extension $g \in B_{\Lip_0(I)}$ of $f \circ \varphi^{-1}$. By absolute continuity of $g$ it follows that
\begin{align*}
 \duality{m,f} &= f(p) - \sum_{n=1}^\infty (f(b_n)-f(a_n))\\
 &= g(\varphi(p)) - \sum_{n=1}^\infty (g(\varphi(b_n))-g(\varphi(a_n)))\\
 &= \int_0^{\varphi(p)} g'\,d\lambda - \sum_{n=1}^\infty \int_{\varphi(a_n)}^{\varphi(b_n)} g'\,d\lambda = \int_{\varphi(C)} g'\,d\lambda \leq \alpha\cdot\norm{g'}_\infty \leq \alpha.
\end{align*}
Therefore $\norm{m} \leq \alpha$.
Next, consider the function $h\in B_{\Lip_0(I)}$ given by $h(t)=\lambda([0,t]\cap\varphi(C))$, and let $H=h\circ\varphi\in B_{\Lip_0(M)}$. Then $H(a_n)=H(b_n)$ for any $n\in\NN$ and
$$
\duality{m,H} = H(p) - \sum_{n=1}^\infty (H(b_n)-H(a_n)) = H(p) = \alpha ,
$$
so $\norm{m}=\alpha$ and $H\in\norming{m}$.

Now fix $\mu\in\opr{\bwt{M}}$ such that $\dual{\Phi}\mu=m$. Let $(x,y) \in \wt{M}$, and consider two cases. If $\varphi(x) \leq \varphi(y)$ then
$$
\Phi H(x,y) = \frac{h(\varphi(x))-h(\varphi(y))}{d(x,y)} \leq 0
$$
because $h$ is an increasing function. On the other hand, if $\varphi(x) > \varphi(y)$ then
$$
\Phi H(x,y) = \frac{h(\varphi(x))-h(\varphi(y))}{d(x,y)} = \frac{\lambda([\varphi(y),\varphi(x)] \cap \varphi(C))}{d(x,y)} < \frac{\varphi(x)-\varphi(y)}{d(x,y)} \leq 1,
$$
because $\varphi(C)$ is nowhere dense and $\varphi$ is $1$-Lipschitz. Thus $\Phi H(x,y) < 1$ whenever $(x,y) \in \wt{M}$. By Lemma \ref{lm:norming_phi1}, we know that $\Phi H(\zeta)=1$ for all $\zeta \in \supp(\mu)$, so this ends the proof.
\end{proof}

\begin{question}
Can ``isometric copy'' be replaced by ``bi-Lipschitz copy'' in the statement of Theorem \ref{th:isometric_fat_cantor_example}?
\end{question}

\begin{remark}\label{rm:no_weak_CIMs}
There is another, non-constructive, method of producing examples of $m \in \lipfree{M}$ satisfying $\supp(\mu) \cap \wt{M}=\varnothing$ for any optimal representation $\mu$ of $m$. Let $\mathrm{LipSNA}_0(M)$ denote the set of $f \in \Lip_0(M)$ that are ``strongly norm-attaining'', meaning that the supremum defining their Lipschitz constant is attained, i.e. $\Phi f(x,y) =\lipnorm{f}$ for some $(x,y) \in \wt{M}$. It is known that there exist metric spaces $M$ for which $\mathrm{LipSNA}_0(M)$ is not dense in $\Lip_0(M)$; see e.g. \cite{Chiclana, CGMRZ, Godefroy_survey, KMS}. On the other hand, by the Bishop-Phelps theorem, the set $\mathrm{LipNA}_0(M)$ of $f \in \Lip_0(M)$ that are norm-attaining when considered as elements of $\lipfree{M}^*$, i.e. $\duality{m,f}=\lipnorm{f}$ for some $m \in S_{\lipfree{M}}$, is dense in $\Lip_0(M)$. Now choose $M$ such that there exists $f \in (\mathrm{LipNA}_0(M) \cap S_{\Lip_0(M)})\setminus\mathrm{LipSNA}_0(M)$, and let $\duality{m,f}=1$ for some $m \in S_{\lipfree{M}}$. Given any optimal representation $\mu$ of $m$, we must have $\supp(\mu) \cap \wt{M}=\varnothing$ by Lemma \ref{lm:norming_phi1} and the fact that $\Phi f(x,y) < \lipnorm{f} = 1$ for all $(x,y) \in \wt{M}$.
\end{remark}

As a second example, we show that there exist functionals that can be expressed as a convex integral of molecules but not as a convex series of molecules, not even failing condition \eqref{eq:convex_series_finite_sum} as in Example \ref{ex:lebesgue_measure_is_convex_series}.
The example will be built in a snowflake space. Recall that given a metric space $(M,d)$ and a real number $\theta\in (0,1)$, the \emph{snowflake} $M^\theta$ is the metric space $(M,d^\theta)$. It is straightforward to check that $d^\theta$ is indeed a metric on $M$.

\begin{proposition}
\label{pr:integral_not_series}
Let $I=[0,1]\subset\RR$ and $\theta\in (0,1)$. Then there is an element of $\lipfree{I^\theta}$ that can be written as a convex integral of molecules but not as a convex series of molecules.
\end{proposition}

\begin{proof}
We will denote $\omega(t)=t^\theta$, so that the metric on $I^\theta$ is given by $d^\theta(x,y)=\omega(\abs{x-y})$.
Let $m\in\lipfree{I^\theta}$ be defined by
$$
\duality{m,f} = \int_0^1 f(x)\,dx
$$
for $f\in\Lip_0(I^\theta)$, that is, $m$ is the functional induced by (the pushforward of) Lebesgue measure on $[0,1]$. By \cite[Proposition 4.4]{AP_measures} and Theorem \ref{th:major_radon_wt}, $m$ belongs to $\lipfree{I^\theta}$ and can be written as a convex integral of molecules. We will prove that it cannot be written as a convex series of molecules. Assume otherwise, that is
$$
m = \sum_{n=1}^\infty a_nm_{p_nq_n} = \sum_{n=1}^\infty a_n\frac{\delta(p_n)-\delta(q_n)}{\omega(\abs{p_n-q_n})}
$$
where $p_n\neq q_n\in I$, $a_n>0$ and $\sum_n a_n=\norm{m}$. Since $m$ is positive, it is normed by the function $x\mapsto d^{\theta}(x,0)=\omega(x)$, that is, by $\omega$, and therefore $\duality{m_{p_nq_n},\omega}=1$ for all $n$. But
$$
\duality{m_{p_nq_n},\omega} = \frac{\omega(p_n)-\omega(q_n)}{\omega(\abs{p_n-q_n})}
$$
is strictly smaller than $1$ when $p_n\neq q_n$ and $q_n\neq 0$, so we must have $q_n=0$ for all $n$ and we may write $m$ as
$$
m = \sum_{n=1}^\infty a_n\frac{\delta(p_n)}{\omega(p_n)} .
$$

Next, for $c,\varepsilon$ such that $0\leq c<c+\varepsilon\leq 1$, consider the function $f_{c,\varepsilon}$ given by
$$
f_{c,\varepsilon}(x) =
\begin{cases}
0 &\text{if } 0\leq x\leq c, \\
\frac{x-c}{\varepsilon} &\text{if }c\leq x\leq c+\varepsilon, \\
1 &\text{if } c+\varepsilon\leq x\leq 1.
\end{cases}
$$
Given any $0\leq x\neq y\leq 1$ we have
$$
\abs{f_{c,\varepsilon}(x)-f_{c,\varepsilon}(y)} \leq \frac{\abs{x-y}}{\varepsilon} \leq \frac{\omega(\abs{x-y})}{\varepsilon}
$$
and therefore $f_{c,\varepsilon}\in\Lip_0(I^\theta)$ with $\norm{f_{c,\varepsilon}}_{\Lip_0(I^\theta)}\leq 1/\varepsilon$. We then have
$$
1-c-\frac{\varepsilon}{2} = \int_0^1 f_{c,\varepsilon}(x)\,dx = \duality{m,f_{c,\varepsilon}} = \sum_{n=1}^\infty a_n\frac{f_{c,\varepsilon}(p_n)}{\omega(p_n)} = \sum_{p_n>c+\varepsilon}\frac{a_n}{\omega(p_n)} + \sum_{c<p_n\leq c+\varepsilon}\frac{a_nf_{c,\varepsilon}(p_n)}{\omega(p_n)}
$$
and thus
$$
\sum_{p_n>c+\varepsilon}\frac{a_n}{\omega(p_n)} \leq 1-c-\frac{\varepsilon}{2} \leq \sum_{p_n>c}\frac{a_n}{\omega(p_n)} \,.
$$
Letting $\varepsilon\to 0$, it follows that
$$
\sum_{p_n>c}\frac{a_n}{\omega(p_n)} = 1-c
$$
for any $0<c<1$. We observe that this equality holds automatically for $c=1$ also.

Now fix $k \in \NN$ and let $c<p_k$. Then
$$
\sum_{c<p_n\leq p_k}\frac{a_n}{\omega(p_n)} = \sum_{p_n>c}\frac{a_n}{\omega(p_n)} - \sum_{p_n>p_k}\frac{a_n}{\omega(p_n)} = (1-c)-(1-p_k) = p_k - c.
$$
Letting $c\nearrow p_k$ yields
$$
0 = \sum_{p_n=p_k}\frac{a_n}{\omega(p_n)} \geq \frac{a_k}{\omega(p_k)},
$$
which is impossible.
\end{proof}

Note that the proof is valid almost verbatim, more generally, for the metric space $I^\omega = (I,\omega\circ d)$ where $\omega$ is any strictly concave \emph{gauge} or \emph{distortion function} as defined in \cite[Section 2.6]{Weaver2} or \cite[Section 3]{Kalton04}.

\section{Lipschitz extensions and metric alignment}
\label{sec:extensions}

In this section, we briefly digress from the main topic in order to prove a general, standalone result about extensions of real-valued Lipschitz functions and their relationship with alignment in the metric space.
In Section \ref{sec:extreme}, this result will be combined with our previous findings on norm-optimal de Leeuw representations in order to obtain information about the extremal structure of Lipschitz-free spaces.

The setting is as follows. Let $A$ be a subset of $M$, and let $f:A\to\RR$ be a $1$-Lipschitz function. The classical McShane extension theorem ensures that $f$ can be extended to a $1$-Lipschitz function $F:M\to\RR$ such that $F\restrict_A=f$. We wish to investigate to what extent the values of $F$ are constrained by the values of $f$ and, in particular, identify the points of $M$ where the values of $F$ are fixed, and the pairs of points of $M$ where the incremental quotient $\Phi F$ is fixed.

For a given $f\in B_{\Lip(A)}$, it is easy to see that there are in fact a smallest and a largest $1$-Lipschitz extension to $M$, usually called the \emph{sup-convolution} and \emph{inf-convolution} of $f$, respectively, and given by
\begin{align*}
E_A^-f(x) &= \sup_{p\in A}\set{f(p)-d(p,x)} \\
E_A^+f(x) &= \inf_{p\in A}\set{f(p)+d(p,x)}
\end{align*}
for $x\in M$. Any $F\in B_{\Lip(M)}$ such that $F\restrict_A=f$ satisfies $E_A^-f\leq F\leq E_A^+f$ pointwise. Conversely, for any $x\in M$ and any $\alpha$ such that $E_A^-f(x)\leq\alpha\leq E_A^+f(x)$ there is an extension $F\in B_{\Lip(M)}$ of $f$ such that $F(x)=\alpha$; to see this, simply consider the functions $t\cdot E_A^+f + (1-t)\cdot E_A^-f$ for $t\in [0,1]$. In particular, if $E_A^-f(x)=E_A^+f(x)$ for some $x$, then all extensions take the same value at $x$. Let us see exactly when this phenomenon takes place. In the following lemma, we will consider sets 
$$
[p,q]_{\varepsilon} = \set{x\in M \,:\, d(p,x)+d(x,q)<d(p,q)+\varepsilon}
$$
for $p,q\in M$ and $\varepsilon>0$, which are open neighbourhoods of metric segments $[p,q]$.

\begin{lemma}
\label{lm:ext1p}
Let $A\subset M$, $f\in S_{\Lip(A)}$, and $x\in M$. Then the following are equivalent:
\begin{enumerate}[label={\upshape{(\roman*)}}]
\item $E_A^-f(x)=E_A^+f(x)$,
\item all $1$-Lipschitz extensions of $f$ to $M$ take the same value at $x$,
\item $x$ belongs to the set
\begin{equation}
\label{eq:saf_general}
\mathscr{S}(A,f) = \bigcap_{\varepsilon>0} \bigcup \set{ [p,q]_\varepsilon \,:\, p,q\in A\text{ and }f(p)-f(q)>d(p,q)-\varepsilon }.
\end{equation}
\end{enumerate}
If $A$ is compact, then we may write simply
\begin{equation}
\label{eq:saf_compact}
\mathscr{S}(A,f) = \bigcup \set{ [p,q] \,:\, p,q\in A\text{ and }f(p)-f(q)=d(p,q) } .
\end{equation}
\end{lemma}

Note that the set $\mathscr{S}(A,f)$ contains $A$ and is always closed e.g. by condition (i).

\begin{proof}
First, we prove the equivalences in the general case.

(i)$\Leftrightarrow$(ii) has been established in the previous discussion.

(i)$\Rightarrow$(iii): Suppose that $E_A^-f(x)=E_A^+f(x)$, and fix $\varepsilon>0$. Then there exist $p,q\in A$ such that
$$
f(p)-d(p,x)+\tfrac{1}{2}\varepsilon > E_A^-f(x) = E_A^+f(x) > f(q)+d(q,x)-\tfrac{1}{2}\varepsilon
$$
so
$$
d(p,q) \leq d(p,x)+d(q,x) < f(p)-f(q)+\varepsilon \leq d(p,q)+\varepsilon .
$$
This shows that $x$ is contained in \eqref{eq:saf_general}.

(iii)$\Rightarrow$(i): Suppose that $x$ belongs to \eqref{eq:saf_general}, and fix $\varepsilon>0$. Then there exist $p,q\in A$ such that
$$
f(p)-f(q) > d(p,q)-\varepsilon > d(p,x)+d(q,x)-2\varepsilon ,
$$
that is,
$$
E_A^-f(x) \geq f(p)-d(p,x) > f(q)+d(q,x)-2\varepsilon \geq E_A^+f(x) - 2\varepsilon .
$$
Letting $\varepsilon\to 0$ we obtain (i).

Finally, we establish the identity \eqref{eq:saf_compact} for compact $A$. Notice that the right-hand side of \eqref{eq:saf_compact} is always contained in $\mathscr{S}(A,f)$. On the other hand, for every $x\in\mathscr{S}(A,f)$ and $n\in\NN$ we may choose $p_n,q_n\in A$ with $x\in [p_n,q_n]_{1/n}$ and $f(p_n)-f(q_n)>d(p_n,q_n)-\frac{1}{n}$. If $A$ is compact then we may extract a subsequence such that $p_n\to p\in A$ and $q_n\to q\in A$, and then we have $x\in [p,q]$ and $f(p)-f(q)=d(p,q)$, that is, $x$ belongs to the right-hand side of \eqref{eq:saf_compact}.
\end{proof}

Going further, let us see when the values of $f$ at $A$ are enough to constrain any extension $F$ to having a fixed incremental quotient $\Phi F(x,y)$ between two points $x,y\in M$ which do not necessarily belong to $A$. This is obviously the case when $F(x)$ and $F(y)$ are individually constrained to fixed values. As it turns out, that is the only situation where this can happen.

\begin{lemma}
\label{lm:ext2p}
Let $A\subset M$, $f\in S_{\Lip(A)}$, and $x\neq y\in M$. Suppose that $F(x)-F(y)$ takes the same value for all $F\in S_{\Lip(M)}$ such that $F\restrict_A=f$. Then $x,y\in\mathscr{S}(A,f)$.
\end{lemma}

\begin{proof}
If exactly one of the points $x,y$, say $x$, belongs to $\mathscr{S}(A,f)$, then $F(x)-F(y)$ can take at least the two different values $E_A^+f(x)-E_A^+f(y)$ and $E_A^-f(x)-E_A^-f(y)=E_A^+f(x)-E_A^-f(y)$. So assume for a contradiction that neither of them belong to $\mathscr{S}(A,f)$ but $F(x)-F(y)$ takes the same value $\beta\in\RR$ for all $1$-Lipschitz extensions $F$. Fix $\alpha$ with $E_A^-f(x)<\alpha<E_A^+f(x)$. Then the function $h:A\cup\set{x}\to\RR$ defined by $h\restrict_A=f$ and $h(x)=\alpha$ is $1$-Lipschitz, therefore it admits $1$-Lipschitz extensions to $A\cup\set{x,y}$. The smallest and largest such extensions are given by
\begin{align*}
E_{A\cup\set{x}}^-h(y) &= \sup_{p\in A\cup\set{x}}\set{h(p)-d(p,y)} = \max\set{E_A^-f(y),\alpha-d(x,y)} \\
E_{A\cup\set{x}}^+h(y) &= \inf_{p\in A\cup\set{x}}\set{h(p)+d(p,y)} = \min\set{E_A^+f(y),\alpha+d(x,y)} ,
\end{align*}
and both of them are extensions of $f$ as well. By assumption
$$
E_{A\cup\set{x}}^-h(x) - E_{A\cup\set{x}}^-h(y) = \beta = E_{A\cup\set{x}}^+h(x) - E_{A\cup\set{x}}^+h(y)
$$
and therefore $E_{A\cup\set{x}}^-h(y)=E_{A\cup\set{x}}^+h(y)$. But we are also assuming $E_A^-f(y)<E_A^+f(y)$, so this implies that at least one of these two conditions must hold:
\begin{itemize}
\item $E_A^+f(y) = \alpha-d(x,y)$
\item $E_A^-f(y) = \alpha+d(x,y)$
\end{itemize}
In the former case, we get $\alpha = E_A^+f(y)+d(x,y) \geq E_A^+f(x)$, whereas in the latter case we get $\alpha\leq E_A^-f(x)$ similarly. Both conclusions contradict our choice of $\alpha$.
\end{proof}

In particular, if $A$ is compact then the incremental quotients of Lipschitz extensions from $A$ to $M$ can only be fixed between points whose coordinates are metrically aligned with points of $A$, due to \eqref{eq:saf_compact}.

We will now focus on the case where the fixed incremental quotient is the maximum possible value $\lipnorm{f}=1$. Let us consider a slightly more general setting, in which we are not restricted to pairs of points $(x,y)$ in $\wt{M}$ but we also consider ``generalised incremental quotients'', that is, evaluations of $\Phi F$ on points of $\bwt{M}$. We fix the notation
\begin{equation}
\label{eq:daf}
\mathscr{D}(A,f) = \set{\zeta\in\bwt{M} \,:\, \text{$\Phi F(\zeta)=1$ for every $
F\in S_{\Lip(M)}$ such that $F\restrict_A=f$} } . 
\end{equation}
Note that this is always a compact subset of $\bwt{M}$, and the situation described in Lemma \ref{lm:ext2p} corresponds to elements of $\mathscr{D}(A,f)$ both of whose coordinates belong to $M$ and are different, i.e. elements of $\mathscr{D}(A,f)\cap\wt{M}$. We will now extend Lemma \ref{lm:ext2p} to cover the situation where both coordinates are equal and belong to $M$, i.e. for elements of $\mathscr{D}(A,f)$ at the ``diagonal'' of $\bwt{M}$. In order to study that case, we first need to extend the lemma from pairs of points to pairs of regions, but this requires a stronger assumption of compactness.

\begin{lemma}
\label{lm:extdiag_claim}
Let $A$ be a compact subset of $M$ and $f\in S_{\Lip(A)}$. Suppose that $p\in M\setminus\mathscr{S}(A,f)$. Then there is $F\in S_{\Lip(M)}$ with $F\restrict_A=f$, and $\eta>0$ and a neighbourhood $U$ of $p$ such that $\abs{\Phi F(x,y)}\leq 1-\eta$ for all $x\in A$ and $y\in U$.
\end{lemma}

\begin{proof}
Fix $x\in A$ for now. By Lemma \ref{lm:ext2p} we have $(x,p)\notin\mathscr{D}(A,f)$, so there exists an extension $F_x^1\in S_{\Lip(M)}$ of $f$ to $M$ such that $\Phi F_x^1(x,p)<1$. Similarly $(p,x)\notin\mathscr{D}(A,f)$ so there is an extension $F_x^2\in S_{\Lip(M)}$ with $\Phi F_x^2(p,x)<1$, i.e. $\Phi F_x^2(x,p)>-1$. Thus $F_x=\frac{1}{2}(F_x^1+F_x^2)$ is a $1$-Lipschitz extension of $f$ such that $\abs{\Phi F_x(x,p)}<1$. By continuity, there are $\eta_x>0$ and neighbourhoods $U_x$ of $p$ and $V_x$ of $x$ such that $\abs{\Phi F_x(x',y')}\leq 1-\eta_x$ whenever $x'\in V_x$ and $y'\in U_x$.

Since $A$ is compact, we may find a finite set $E\subset A$ such that $A\subset\bigcup_{x\in E}V_x$. Let
$$
F=\frac{1}{\abs{E}}\sum_{x\in E}F_x
$$
and $U=\bigcap_{x\in E}U_x$. Then $F\in S_{\Lip(M)}$ is also a $1$-Lipschitz extension of $f$ and $U$ is a neighbourhood of $p$. For any $x\in A$ we may find $x'\in E$ such that $x\in V_{x'}$ and so $\abs{\Phi F_{x'}(x,y)}\leq 1-\eta_{x'}$ for any $y\in U$. Thus the required conditions are satisfied with $\eta=\frac{1}{\abs{E}}\min\set{\eta_x\,:\,x\in E}$.
\end{proof}

With this technical result at hand, we may extend Lemma \ref{lm:ext2p} to the diagonal:

\begin{lemma}
\label{lm:extdiag}
Let $A$ be a compact subset of $M$ and $f\in S_{\Lip(A)}$. If $p\in M$ and $\zeta\in\mathscr{D}(A,f)$ are such that $\pp(\zeta)=(p,p)$, then $p\in\mathscr{S}(A,f)$.
\end{lemma}

\begin{proof}
Let $p\in M$ and let $\zeta\in\mathscr{D}(A,f)$ satisfy $\pp(\zeta)=(p,p)$. 
Suppose that $p\notin\mathscr{S}(A,f)$, in order to derive a contradiction. Apply Lemma \ref{lm:extdiag_claim} to get a $1$-Lipschitz extension $g$ of $f$, a neighbourhood $U$ of $p$, and $\eta>0$ such that $\abs{\Phi g(x,y)}\leq 1-\eta$ for all $x\in A$ and $y\in U$.
We may assume that $U$ is bounded and $r:=d(A,U)$ is positive. Now fix $\delta\in(0,1)$ such that $\delta\norm{g\restrict_U}_\infty\leq\eta r$, and define a function $h$ on $A\cup U$ by
$$
h(x) = \begin{cases}
g(x) &\text{if $x\in A$,} \\
(1-\delta)g(x) &\text{if $x\in U$.}
\end{cases}
$$
We claim that $h$ is $1$-Lipschitz. Indeed, it is clear that $\abs{h(x)-h(y)}\leq d(x,y)$ if $x,y$ both belong to $A$ or to $U$, and if $x\in A$ and $y\in U$ then
$$
\abs{h(x)-h(y)} \leq \abs{g(x)-g(y)}+\delta\abs{g(y)} \leq (1-\eta)d(x,y) + \eta r \leq d(x,y)
$$
as well. Thus $h$ can be extended to a $1$-Lipschitz extension of $f$. Since $h$ is $(1-\delta)$-Lipschitz in a neighbourhood of $p$, we get that $\abs{\Phi h(\zeta)}\leq 1-\delta$ and so $\zeta\notin\mathscr{D}(A,f)$. This contradiction ends the proof.
\end{proof}

Finally, we prove a version of Lemma \ref{lm:ext2p} for elements of $\mathscr{D}(A,f)$ having only one coordinate in $M$ and the other in $\beta M\setminus M$.

\begin{lemma}
\label{lm:ext_mixed_coordinates}  
Let $A$ be a compact subset of $M$ and $f\in S_{\Lip(A)}$. If $p\in M$ and $\zeta\in\mathscr{D}(A,f)$ satisfy $\pp_s(\zeta)=\{p,\xi\}$ for some $\xi\in\beta M\setminus M$, then 
$p\in\mathscr{S}(A,f).$
\end{lemma}

\begin{proof}
Let $p\in M$, $\xi\in\beta M\setminus M$ and $\zeta\in\mathscr{D}(A,f)$ be such that $\pp(\zeta)=(p,\xi)$, and suppose for a contradiction that $p\notin\mathscr{S}(A,f)$. First observe that $\xi\in\rcomp{M}$ (and therefore also $|F(\xi)|<\infty$ for any Lipschitz extension $F$ of $f$ to $M$). Indeed, let $(p_i,x_i)\in\wt{M}$ such that $\zeta=\lim_i(p_i,x_i)$ in $\bwt{M}$ and $p=\lim_ip_i$, $\xi=\lim_ix_i$ in $\beta M$. Because $\zeta \in\mathscr{D}(A,f)$, we have 
$$1=\Phi(E_A^+f)(\zeta)=\lim_i\frac{E_A^+f(p_i)-E_A^+f(x_i)}{d(p_i,x_i)}.$$ Thus eventually 
$$d(p_i,x_i)<2(E_A^+f(p_i)-E_A^+f(x_i))<2(E_A^+f(p)+1-\inf_{y\in A}f(y))<\infty,$$ where the finiteness follows from the boundedness of $A$. Hence $d(\zeta)=d(p,\xi)<\infty$, where $d(\zeta)$ is the evaluation at $\zeta$ of the continuous extension of $d(\cdot,\cdot)$ from $\wt{M}$ to $\bwt{M}$ and $d(p,\xi)$ is the evaluation at $\xi$ of the continuous extension of $d(p,\cdot)$ from $M$ to $\beta M$. So $\xi\in\rcomp{M}$ as claimed.

Let $d(\xi,A\cup\{p\})$ denote the evaluation at $\xi$ of the continuous extension of $d(\cdot,A\cup\{p\})$ to $\beta M$. This value is positive by the compactness of $A$ and finite because $\xi\in\rcomp{M}$. Choose $0<t<d(\xi,A\cup\{p\})$ and denote 
$$V=\set{x\in M\,:\, E_A^+f(x)<E_A^+f(p) \text{ and } d(x,A\cup\{p\})>t}.$$ 
It follows from the assumption on $\zeta$ and the choice of $t$ that $\xi\in\cl{V}^{\beta M}$. Because $p\notin\mathscr{S}(A,f)$, we also have that $E_A^-f(p)<E_A^+f(p)$ by Lemma \ref{lm:ext1p}. Take
$$
0 < \lambda < \min\set{\frac{t}{E_A^+f(p)-E_A^-f(p)},1}
$$
and define a function $h:A\cup V \cup \{p\} \to \RR$ by
\[
h(x)=\begin{cases} f(x) & \text{if }x \in A,\\
E_A^+f(x) & \text{if }x \in V,\\
E_A^+f(p)-\lambda(E_A^+f(p)-E_A^-f(p)) & \text{if }x=p.
\end{cases}
\]
Then $h$ is $1$-Lipschitz because $E_A^+f$ and $E_A^-f$ are $1$-Lipschitz extensions of $f$ to $M$, and for every $x\in V$ we have
\begin{align*}
h(x) - h(p)&=E_A^+f(x)-E_A^+f(p)+\lambda(E_A^+f(p)-E_A^-f(p))\\
&> E_A^+f(x)-E_A^+f(p)\geq -d(x,p)
\end{align*}
and
\begin{align*}
 h(x) - h(p)&=E_A^+f(x)-E_A^+f(p)+\lambda(E_A^+f(p)-E_A^-f(p))\\
&< E_A^+f(x)-E_A^+f(p)+t<t<d(x,p). 
\end{align*}
So, $h$ can be extended to a $1$-Lipschitz extension of $f$ to $M$. But
$$
\Phi h(\zeta)=\frac{h(p)-h(\xi)}{d(\zeta)}=\frac{E_A^+f(p)-E_A^+f(\xi)}{d(\zeta)}-\lambda\frac{E_A^+f(p)-E_A^-f(p)}{d(\zeta)}<\Phi (E_A^+f)(\zeta) = 1,
$$
which contradicts the assumption that $\zeta\in\mathscr{D}(A,f)$.

If $\pp(\zeta)=(\xi,p)$ then we proceed analogously, replacing $E_A^+f$ with $E_A^-f$.
\end{proof}

Combining Lemmas \ref{lm:ext2p}, \ref{lm:extdiag} and \ref{lm:ext_mixed_coordinates}, we obtain the following general statement:

\begin{proposition}
\label{pr:extdiag}
Let $A\subset M$ and $f\in S_{\Lip(A)}$. If $A$ is compact, then $\pp_s(\mathscr{D}(A,f))\cap M\subset\mathscr{S}(A,f)$. 
\end{proposition}

\section{Applications to extremal structure}
\label{sec:extreme}

We will now apply our previous results to analyze the extremal structure of the unit ball of $\lipfree{M}$. The main open problem in this direction is the characterisation of the extreme points of $B_{\lipfree{M}}$. More specifically, the unproven conjecture is that all extreme points must be elementary molecules. The description of those molecules that are extreme points is already known, see \cite[Theorem 1.1]{AP_rmi} (or Corollary \ref{cr:extreme_molecules} below). The conjecture is known to hold when $M$ is proper \cite{Aliaga} or a subset of an $\RR$-tree \cite{APP}. It is also easy to see that any extreme point that is a convex series of molecules must be a molecule itself \cite[Remark 3.4]{APPP}. We will now extend that statement to convex integrals of molecules. In fact, we will prove a slightly stronger statement.

\begin{theorem}
\label{th:extreme_wt}
Let $m$ be an extreme point of $B_{\lipfree{M}}$. Suppose that $m$ has an optimal de Leeuw representation $\mu$ such that $\mu(\wt{M})>0$. Then $m$ is an elementary molecule.

In particular, the conclusion holds if $m$ is a convex integral of molecules.
\end{theorem}

Our proof relies on the following lemma from \cite{Aliaga}. We include its short proof for reference.

\begin{lemma}[{\cite[Lemma 10]{Aliaga}}]
\label{lm:lemma_10_aliaga}
Let $m$ be an extreme point of $B_{\lipfree{M}}$, and let $\mu$ be an optimal de Leeuw representation of $m$. If $\lambda\in\meas{\bwt{M}}$ is such that $0\leq\lambda\leq\mu$ and $\dual{\Phi}\lambda\in\lipfree{M}$, then $\dual{\Phi}\lambda=\norm{\lambda}\cdot m$.
\end{lemma}

\begin{proof}
Note that $\lambda$ and $\mu-\lambda$ belong to $\opr{\bwt{M}}$ by Proposition \ref{pr:opr_facts}(c). If either of them is $0$ then the lemma is trivial. Otherwise
$$
m = \dual{\Phi}\lambda + \dual{\Phi}(\mu-\lambda) = \norm{\lambda}\cdot\dual{\Phi}\pare{\frac{\lambda}{\norm{\lambda}}} + \norm{\mu-\lambda}\cdot\dual{\Phi}\pare{\frac{\mu-\lambda}{\norm{\mu-\lambda}}}
$$
is a convex combination of elements of $B_{\lipfree{M}}$ by Proposition \ref{pr:opr_facts}(b), hence $m=\dual{\Phi}(\lambda/\norm{\lambda})$ by extremality.
\end{proof}

\begin{proof}[Proof of Theorem \ref{th:extreme_wt}]
By assumption $\mu\restrict_{\wt{M}}\neq 0$, and $\dual{\Phi}(\mu\restrict_{\wt{M}})\in\lipfree{M}$ by Proposition \ref{pr:wt_bochner}. Thus Lemma \ref{lm:lemma_10_aliaga} implies that $m=\dual{\Phi}(\mu\restrict_{\wt{M}}\cdot\norm{\mu\restrict_{\wt{M}}}^{-1})$. So, by substituting $\mu$ with $\mu\restrict_{\wt{M}}\cdot\norm{\mu\restrict_{\wt{M}}}^{-1}\in\opr{\wt{M}}$, we may assume from the start that $\mu$ is concentrated on $\wt{M}$.

Fix $(x,y)\in\supp(\mu)\cap\wt{M}$. Let $U,V$ be disjoint closed neighbourhoods of $x,y$, and put $\lambda=\mu\restrict_{U\times V}$. Then $\lambda\in\opr{\wt{M}}$ by Proposition \ref{pr:opr_facts}(d), and $\dual{\Phi}\lambda\in\lipfree{M}$. Moreover $\norm{\lambda}=\lambda(U\times V)>0$. By Lemma \ref{lm:lemma_10_aliaga} we conclude $m=\dual{\Phi}\lambda/\norm{\lambda}$, and so Proposition \ref{pr:aliaga_lemma_8} implies that $\supp(m)\subset\pp_s(\supp(\lambda))\subset U\cup V$. Since this is true for every choice of $U$ and $V$, we finally get $\supp(m)\subset\set{x,y}$. Thus $m$ is supported on two points at most, and it is then easy to see that it must be an elementary molecule (see e.g. \cite[Theorem 1.1]{AP_rmi}).
\end{proof}

Combining Theorem \ref{th:majorizable_convex_integrals} and Corollaries \ref{cr:rd_cs} and \ref{cr:ud_cs} with Theorem \ref{th:extreme_wt}, we obtain new cases of the extreme point conjecture.

\begin{corollary}
\label{cr:extreme_rud}
If an extreme point of $B_{\lipfree{M}}$ is majorisable on annuli, then it is an elementary molecule.

In particular, if $M$ is uniformly discrete or radially discrete, then all extreme points of $B_{\lipfree{M}}$ are elementary molecules.
\end{corollary}

This result characterises all extreme points of $B_{\lipfree{M}}$ for uniformly discrete or radially discrete $M$ when paired together with \cite[Theorem 1.1]{AP_rmi} (or alternatively with \cite[Proposition 5.1]{GPPR}, where the description of extreme molecules for bounded and uniformly discrete $M$ was first obtained). Corollary \ref{cr:extreme_rud} also extends \cite[Theorem 3.10]{APPP}, which asserts that any extreme point that can be written as a positive element plus a finitely supported element must be a molecule.

\medskip

Our second result relates to the structure of faces of $B_{\lipfree{M}}$. Recall that a \textit{face}, or \textit{extreme subset}, of a convex set $C$ is a non-empty convex subset $E\subset C$ with the property that $x,y\in E$ whenever $x,y\in C$ and $tx+(1-t)y\in E$ for some $t\in (0,1)$. Extreme points are thus precisely the elements of faces that are singletons.
Relying on the study of Lipschitz extensions carried out in Section \ref{sec:extensions} we can show that, under an assumption of compactness, the support of any internal point in a face of $S_{\lipfree{M}}$ constrains the supports in the rest of the face through an alignment condition.

\begin{theorem}
\label{th:alignment_compact}
Suppose that $S_{\lipfree{M}}$ contains a linear segment $I$, and some element $m$ in the interior of $I$ has compact support. Then
$$
\supp(m') \subset \bigcup\set{[p,q] \,:\, p,q\in\supp(m)\cup\set{0}}
$$
for every $m'\in I$.
\end{theorem}

Theorem \ref{th:alignment_compact} generalises \cite[Theorem 2.4]{AFGLZ}, where $M$ is assumed to be finite and $m$ is a molecule.
For the proof, we will use a restatement of Lemma \ref{lm:norming_phi1} using the notation
\begin{equation}
\label{eq:dm}
\mathscr{D}(m) = \set{\zeta\in\bwt{M} \,:\, \text{$\Phi f(\zeta)=1$ for each $f\in\norming{m}$} }
\end{equation}
for $m\in\lipfree{M}$ or $\dual{\Lip_0(M)}$.

\begin{lemma}
\label{lm:supp_pm}
For any $\mu\in\opr{\bwt{M}}$ we have $\supp(\mu)\subset\mathscr{D}(\dual{\Phi}\mu)$. Moreover, if $m=\dual{\Phi}\mu\in\lipfree{M}$ then
$$
\mathscr{D}(m) = \bigcap_{f\in\norming{m}} \mathscr{D}(\supp(m)\cup\set{0},f\restrict_{\supp(m)\cup\set{0}}) \,.
$$
\end{lemma}

\begin{proof}
Let $m=\dual{\Phi}\mu$. The first statement is immediate from Lemma \ref{lm:norming_phi1} as we have $\Phi f(\zeta)=1$ for all $f\in\norming{m}$ and $\zeta\in\supp(\mu)$. For the second one, it suffices to notice that the value of $\duality{m,f}$ depends only on $f\restrict_{\supp(m)}$. Therefore, if $g\in B_{\Lip_0(M)}$ agrees with $f\in\norming{m}$ on $\supp(m)$, and thus also on $\supp(m)\cup\set{0}$, then $g\in\norming{m}$ as well.
\end{proof}

\begin{proof}[Proof of Theorem \ref{th:alignment_compact}]
Fix $m'\in I$. First notice that $\norming{m}\subset\norming{m'}$. Indeed, since $m$ is an interior point of $I$, we have $m+t(m'-m)\in I$ for $\abs{t}<\varepsilon$ for some $\varepsilon>0$. For any $f\in\norming{m}$ we then have
$$
1 = \norm{m+t(m'-m)} \geq \duality{m+t(m'-m),f} = 1+t\duality{m'-m,f}
$$
for $-\varepsilon<t<\varepsilon$, which implies $\duality{m'-m,f}=0$ and thus $f\in\norming{m'}$ as claimed. It follows that $\mathscr{D}(m')\subset\mathscr{D}(m)$, using the notation \eqref{eq:dm}. Now let $\mu\in\opr{\bwt{M}}$ be an optimal representation of $m'$. Then, applying Proposition \ref{pr:aliaga_lemma_8} and Lemma \ref{lm:supp_pm} we obtain
$$
\supp(m') \subset \pp_s(\supp(\mu)) \subset \pp_s(\mathscr{D}(m')) \subset \pp_s(\mathscr{D}(m)) .
$$
Write $A=\supp(m)\cup\set{0}$ and fix some $f\in\norming{m}$. Then $\mathscr{D}(m)\subset\mathscr{D}(A,f\restrict_A)$ by Lemma \ref{lm:supp_pm} and, since $A$ is compact by assumption, Proposition \ref{pr:extdiag} yields
$$
\supp(m') \subset \pp_s(\mathscr{D}(A,f\restrict_A)) \cap M \subset \mathscr{S}(A,f\restrict_A)
$$
which completes the proof using \eqref{eq:saf_compact}.
\end{proof}

In particular, we have the following. Recall that a metric space is \emph{concave} if it contains no (non-trivial) metric triple, i.e. if it satisfies $d(x,y)<d(x,z)+d(z,y)$ for all distinct points $x,y,z\in M$.

\begin{corollary}
If $M$ is compact and concave then all internal points of a face of $S_{\lipfree{M}}$ have the same support, up to possibly the base point.
\end{corollary}

We can also obtain the main result from \cite{AP_rmi} as a particular case of Theorem \ref{th:alignment_compact}.

\begin{corollary}[cf. {\cite[Theorem 1.1]{AP_rmi}}]
\label{cr:extreme_molecules}
Let $x\neq y\in M$. Then $m_{xy}$ is an extreme point of $B_{\lipfree{M}}$ if and only if $[x,y]=\set{x,y}$.
\end{corollary}

\begin{proof}
The forward implication is immediate from \eqref{eq:molecule_splitting}. Now assume $[x,y]=\set{x,y}$ and suppose that $y=0$ without loss of generality. If $m_{xy}$ is not an extreme point, then it is contained in the interior of a segment $I\subset S_{\lipfree{M}}$. By Theorem \ref{th:alignment_compact}, every $m\in I$ satisfies $\supp(m)\subset [x,y]$, hence $\supp(m)=\set{x}$ and this forces $m=\pm m_{xy}$, which is clearly not possible.
\end{proof}

It should now become apparent that our approach in this paper is, to an extent, a generalisation of the one followed in \cite{AP_rmi}. Indeed, Section \ref{sec:extensions} herein replaces \cite[Lemma 4.1]{AP_rmi}, and Lemma \ref{lm:supp_pm} plays the role of \cite[Lemmas 4.2 and 4.3]{AP_rmi}.

\section*{Acknowledgements}

R. J. Aliaga was partially supported by Grant BEST/2021/080 funded by the Generalitat Valenciana, Spain, and by Grant PID2021-122126NB-C33 funded by MCIN/AEI/10.13039/501100011033 and by ``ERDF A way of making Europe''. E. Perneck\'a was supported by the grant GA\v CR \mbox{22-32829S} of the Czech Science Foundation.

Parts of this research were conducted while the first author visited the Laboratoire de Math\'emat\-iques de Besan\c{c}on in 2021 and the Faculty of Information Technology at the Czech Technical University in Prague in 2022, for which he wishes to express his gratitude. Some initial seeds of this research were formulated during a visit of the third author to the Faculty of Information Technology, Czech Technical University, Prague, in December 2019. The third author would also like to express his gratitude for its hospitality.


\end{document}